\documentclass[11pt]{amsart}

\usepackage{latexsym}
\usepackage{amssymb}
\usepackage{amsmath}
\usepackage{fancybox,color}
\usepackage{amssymb}
\usepackage{amsfonts}
\usepackage{latexsym}
\usepackage{amsmath,systeme}
\usepackage[leqno]{mathtools}

\mathtoolsset{showonlyrefs}

\usepackage{hyperref}

\def\1{\raisebox{2pt}{\rm{$\chi$}}}

\newtheorem{theorem}{Theorem}
\newtheorem{corollary}[theorem]{Corollary}
\newtheorem{lemma}[theorem]{Lemma}
\newtheorem{proposition}[theorem]{Proposition}
\newtheorem{definition}[theorem]{Definition}
\newtheorem{remark}[theorem]{Remark}

\newcommand{\R}{{\mathbb R}}
\newcommand{\E}{{\mathbb E}}

\newcommand{\N}{{\mathbb N}}

\newcommand{\W}{{\mathbb W}}
\renewcommand{\P}{{\mathbb P}}

 \newcommand{\eps}{{\varepsilon}}
 \def\1{\raisebox{2pt}{\rm{$\chi$}}}

\usepackage{enumerate}

\def\vint_#1{\mathchoice%
          {\mathop{\kern 0.2em\vrule width 0.6em height 0.69678ex depth -0.58065ex
                  \kern -0.8em \intop}\nolimits_{\kern -0.4em#1}}%
          {\mathop{\kern 0.1em\vrule width 0.5em height 0.69678ex depth -0.60387ex
                  \kern -0.6em \intop}\nolimits_{#1}}%
          {\mathop{\kern 0.1em\vrule width 0.5em height 0.69678ex
              depth -0.60387ex
                  \kern -0.6em \intop}\nolimits_{#1}}%
          {\mathop{\kern 0.1em\vrule width 0.5em height 0.69678ex depth -0.60387ex
                  \kern -0.6em \intop}\nolimits_{#1}}}
\def\vintslides_#1{\mathchoice%
          {\mathop{\kern 0.1em\vrule width 0.5em height 0.697ex depth -0.581ex
                  \kern -0.6em \intop}\nolimits_{\kern -0.4em#1}}%
          {\mathop{\kern 0.1em\vrule width 0.3em height 0.697ex depth -0.604ex
                  \kern -0.4em \intop}\nolimits_{#1}}%
          {\mathop{\kern 0.1em\vrule width 0.3em height 0.697ex depth -0.604ex
                  \kern -0.4em \intop}\nolimits_{#1}}%
          {\mathop{\kern 0.1em\vrule width 0.3em height 0.697ex depth -0.604ex
                  \kern -0.4em \intop}\nolimits_{#1}}}

\newcommand{\kint}{\vint}

\newcommand{\aveint}[2]{\mathchoice%
          {\mathop{\kern 0.2em\vrule width 0.6em height 0.69678ex depth -0.58065ex
                  \kern -0.8em \intop}\nolimits_{\kern -0.45em#1}^{#2}}%
          {\mathop{\kern 0.1em\vrule width 0.5em height 0.69678ex depth -0.60387ex
                  \kern -0.6em \intop}\nolimits_{#1}^{#2}}%
          {\mathop{\kern 0.1em\vrule width 0.5em height 0.69678ex depth -0.60387ex
                  \kern -0.6em \intop}\nolimits_{#1}^{#2}}%
          {\mathop{\kern 0.1em\vrule width 0.5em height 0.69678ex depth -0.60387ex
                  \kern -0.6em \intop}\nolimits_{#1}^{#2}}}

\newcommand{\half}{{\frac{1}{2}}}
\newcommand{\abs}[1]{\left| #1 \right|}

\newcommand{\ol}{\overline}
\newcommand{\Om}{\Omega}
\newcommand{\I}{\textrm{I}}
\newcommand{\II}{\textrm{II}}

\begin{document}

\title[An elliptic system with two different operators]
{\bf A game theoretical approach for an elliptic system with two different operators
(the Laplacian and the infinity Laplacian)}

\author[A. Miranda and J. D. Rossi]
{Alfredo Miranda and Julio D. Rossi}

\address{Alfredo Miranda and Julio D. Rossi
\hfill\break\indent
Departamento  de Matem{\'a}tica, FCEyN,
Universidad de Buenos Aires,
\hfill\break\indent Pabellon I, Ciudad Universitaria (1428),
Buenos Aires, Argentina.}

\email{{\tt manfredpichota@gmail.com, jrossi@dm.uba.ar}}

\date{}

\begin{abstract} In this paper we find viscosity solutions to an elliptic system governed by two different operators 
(the Laplacian and the infinity Laplacian) using a probabilistic approach. 
We analyze a game that combines the Tug-of-War with Random Walks in two different boars. We show that these value functions converge uniformly
to a viscosity solution of the elliptic system as the step size goes to zero. 

In addition, we show uniqueness for the elliptic system using pure PDE techniques.
\end{abstract}

\maketitle


\section{Introduction}

Our goal in this paper is to describe a probabilistic game whose
value functions approximate viscosity solutions to the following elliptic system:
\begin{equation}
\label{ED1}
 \left\lbrace
\begin{array}{ll}
- \displaystyle \half \Delta_{\infty}u(x) + u(x) - v(x)=0 \qquad &  \ x \in \Omega,  \\[10pt]
-   \displaystyle \frac{\kappa}{2} \Delta v(x) + v(x) - u(x)=0  \qquad & \ x \in \Omega,  \\[10pt]
u(x) = f(x) \qquad & \ x \in \partial \Omega,  \\[10pt]
v(x) = g(x) \qquad & \ x \in \partial \Omega, 
\end{array}
\right.
\end{equation}
here $\kappa >0$ is a constant that can be chosen adjusting the parameters of the game. 
The domain $\Omega \subset \R^N$ is assumed to be bounded and satisfy 
the uniform exterior ball property, that is, there is $\theta > 0$ such that for all $y\in\partial\Om$ there exists a closed ball 
of radius $\theta$ that only touches $\ol{\Om}$ at $y$. This means that, for each $y\in\partial\Om$ there exists a $z_{y}\in \R^{N}\backslash\Om$ such that $\ol{B_{\theta}(z_{y})}\cap\ol{\Om}= \{ y \}$. The boundary data $f$ and $g$ are assumed to be Lipschitz functions.

Notice that this system involves two differential operators, the usual Laplacian 
$$ \Delta \phi = \sum\limits_{i=1}^{N} \partial_{x_{i}x_{i}}\phi$$ and the infinity Laplacian (see \cite{Cran})
 $$
 \Delta_{\infty}\phi =  
 \langle D^{2}\phi \frac{\nabla \phi}{|\nabla \phi |},\frac{\nabla \phi}{| \nabla \phi |}\rangle 
 =\frac{1}{| \nabla \phi |^2}\sum\limits_{i,j=1}^{N} \partial_{x_{i}}\phi \partial_{x_{i}x_{j}}\phi
 \partial_{x_{j}}\phi.$$ 
 
 This system \eqref{ED1} is not variational (there is no associated energy). Therefore, to find solutions
 one possibility is to use monotonicity methods (Perron's argument). Here we will look at the system
 in a different way and to obtain existence of solutions we find an approximation using game theory.
This approach not only gives existence of solutions but it also provide us with a description that
yield some light on the behaveiour of the solutions.  
At this point we observe that we will understand solutions to the system in the viscosity 
sense, this is natural since the infinity Laplacian is not variational (see Section \ref{sect-prelim} for
the precise definition). 

The fundamental works by Doob, Feller, Hunt, Kakutani, Kolmogorov and many others
show the deep connection between classical potential theory and probability theory.
The main idea that is behind this relation is that harmonic functions 
and martingales have something in common: the mean value formulas.
A well known fact 
is that $u$ is harmonic, that is $u$ verifies the PDE $\Delta u =0$, 
if and only if it verifies the mean value property
$
u(x) =
\frac{1}{|B_\varepsilon (x) |} \int_{B_\varepsilon (x) } u(y) \, dy.
$
In fact, we can relax
this condition by requiring that it holds asymptotically
$
u(x) =
\frac{1}{|B_\varepsilon (x) |} \int_{B_\varepsilon (x) } u(y) \, dy + o(\varepsilon^2),
$
as $\varepsilon\to 0$.
The connection between the Laplacian and the Bownian motion or with the limit
of random walks as the step size goes to zero is also well known, see \cite{Kac}. 

The ideas and techniques used for linear equations 
have been extended to cover nonlinear cases as well.
Concerning nonlinear equations, for a mean value property for the $p-$Laplacian (including 
the infinity Laplacian) we refer to \cite{I}, \cite{KMP}, \cite{LM} and \cite{MPR}. For a probabilistic
approximation of the infinity Laplacian
there is a game (called Tug-of-War game in the literature) that was introduced in \cite{PSSW} and 
generalized in several directions to cover other equations, like the $p-$Laplacian, see \cite{TPSS,AS,BR,ChGAR,LPS,MPR,
MPRa,MPRb,Mitake,PS,R}
and the book \cite{BRLibro}.

Now let us describe the game that is associated with \eqref{ED1}.
It is a two-player zero-sum game played in two different bards (two different copies of the set $\Omega \subset \mathbb{R}^N$). Fix a parameter, $\eps>0$ and two final payoff functions 
$ \ol{f},  \ol{g} :\mathbb{R}^N \setminus \Omega \mapsto \mathbb{R}$ (one for each board, $ \ol{f}$ for the first board and $ \ol{g}$ for the second one).
These payoff functions  $ \ol{f}$ and $ \ol{g}$ are just two Lipschitz extensions to $\mathbb{R}^N \setminus \Omega$ of the boundary data $f$ and $g$ that appear in \eqref{ED1}. 
The rules of the game are the following:
the game starts with a token at an initial position $x_0 \in \Omega$, in one of the two boards. 
In the fist board, with probability $1-\eps^2$, 
the players play Tug-of-War as described in \cite{PSSW,MPR} (this game is associated
with the infinity Laplacian). Playing Tug-of-War, the players toss a fair coin and the winner
chooses a new position of the game with the restriction that $x_1 \in B_\eps (x_0)$. When the token is
in the first board, with probability
$\eps^2$ the token jumps to the other board (at the same position $x_0$). In the second board 
with probability $1-\eps^2$ the token is moved at random (uniform probability) to some point 
$x_1 \in B_\eps (x_0)$ and with probability $\eps^2$ the token jumps back to the first board
(without changing the position). 
The game continues until the position of the token leaves the domain and at this 
point $x_\tau$ the first player gets $\ol{f}(x_\tau)$ and the second player $- \ol{f}(x_\tau)$ if they are
playing in the first board while they obtain $\ol{g}(x_\tau)$ and $- \ol{g}(x_\tau)$ if they are
playing in the second board
(we can think that Player $\II$ pays to Player $\I$ the amount given by $\ol{f}(x_\tau)$ or by $\ol{g}(x_\tau)$
according to the board in which the game ends).
This game has a expected value (the best outcome of the game that both players expect to obtain
playing their best,
see Section \ref{sect-game2} for a precise definition). In this case the value of the game
is given by a pair of functions $(u^\eps, v^\eps)$, defined in $\Omega$ 
that depends on the size of the steps, $\eps$. For $x_0 \in \Omega$, the value of $u^\eps (x_0)$
is the expected outcome of the game when it starts at $x_0$ in the first board and $v^\eps (x_0)$ is the expected value starting at $x_0$ in the second board.

Our fist theorem ensures that this game has a well-defined value and that this pair of functions $(u^\eps, v^\eps)$
 verifies a system of equations (called the dynamic programming principle (DPP))
in the literature).  Similar results are proved in \cite{BR,LPS,MPRa,PSSW,R}.

\begin{theorem} \label{teo.dpp2}
The game has value 
$
(u^\eps, v^\eps)
$
that verifies 
\begin{equation}
\label{DPP}
 \left\lbrace
\begin{array}{ll}
 \displaystyle u^{\eps}(x)=\eps^{2}v^{\eps}(x)+(1-\eps^{2})\Big\{\half \sup_{y \in B_{\eps}(x)}u^{\eps}(y) + \half \inf_{y \in B_{\eps}(x)}u^{\eps}(y)
 \Big\} \quad &  \ x \in \Omega,  \\[10pt]
   \displaystyle v^{\eps}(x)=\eps^{2}u^{\eps}(x)+(1-\eps^{2})\kint_{B_{\eps}(x)}v^{\eps}(y)dy  \quad &  \ x \in \Omega,  \\[10pt]
u^{\eps}(x) = \ol{f}(x) \quad & \ x \in \R^{N} \backslash \Omega,  \\[10pt]
v^{\eps}(x) = \ol{g}(x) \quad &  \ x \in \R^{N} \backslash \Omega. 
\end{array}
\right.
\end{equation} 
Moreover, there is a unique solution to \eqref{DPP}.
\end{theorem}

Notice that \eqref{DPP} can be see as a sort of mean value property (or a discretization at scale of size $\eps$) for the system \eqref{ED1}. 
Let see intuitively why the DPP \eqref{DPP} holds. Playing in the first board,
at each step Player $\I$ chooses the next position of the game with probability
$\frac{1-\eps^2}{2}$ and aims to obtain $\inf_{y \in B_{\eps}(x)}u^{\eps}(y)$
 (recall this player seeks to minimize the expected payoff); with probability
$\frac{1-\eps^2}{2}$ it is Player $\II$ who choses and aims to obtain $\sup_{y \in B_{\eps}(x)}u^{\eps}(y)$ and finally
with probability $\eps^2$ the board changes (and therefore $v^\eps (x)$ comes into play). 
Playing in the second board, with probability $1-\eps^2$ the point moves at random (but stays in the second board)
and hence the term $\kint_{B_{\eps}(x)}v^{\eps}(y)dy $ appears, but with probability $\eps^2$ the board is changed and 
hence we have $u^\eps (x)$ in the second equation.
The equations
in the (DPP) follow just by considering all the possibilities.
 Finally, the final payoff at $x \not\in \Omega$ is given by $ \ol{f} (x)$ in the first board and by $ \ol{g}(x)$
 in the second board, giving the exterior conditions $u^{\eps}(x) = \ol{f}(x)$ and $v^{\eps}(x) = \ol{g}(x)$. 

Our next goal is to look for the limit as $\eps \to 0$. 
Our main result in this paper is to show that,
under our regularity conditions on the data ($\partial \Omega$ verifies the uniform exterior ball condition, and
$f$ and $g$ are Lipschitz),
these value functions $u^\eps, v^\eps$ converge uniformly
in $\overline{\Omega}$
to a pair of continuous limits $u,v$ that are characterized as being the unique viscosity solution to \eqref{ED1}.

\begin{theorem} \label{teo.converge2} 
Let $(u^\eps, v^\eps)$ be the values of the game.
Then, there exists a pair of continuous functions in $\overline{\Omega}$, $(u,v)$, such that
\begin{equation*}
u^\eps \to u, \quad \mbox{and} \quad v^\eps \to v, \qquad \mbox{ as } \eps \to 0,
\end{equation*}
uniformly in $\overline{\Omega}$.
Moreover, the limit $(u,v)$ is characterized as the unique viscosity
solution to the system \eqref{ED1}  (with a constant $\kappa=\frac{1}{|B_{1}(0)|}\int_{B_{1}(0)}z_{j}^{2}dz. 
$ that depends only on the dimension).
\end{theorem}

\begin{remark} {\rm
If we impose that the probability of moving random in the second board is $1-K\eps^2$ (and
hence the probability of changing from the second to the first board is $K\eps^2$) with the same computations we obtain 
$$
v^{\eps}(x)=K \eps^{2}u^{\eps}(x)+(1- K \eps^{2})\kint_{B_{\eps}(x)}v^{\eps}(y)dy
$$
 as the second equation in the DPP (the first equation and the exterior data remain unchanged). Passing to the limit we get
 $$
 -   \displaystyle \frac{\kappa}{2 K } \Delta v(x) + v(x) - u(x)=0 ,
 $$
 and hence, choosing $K$, we can obtain any positive constant in front of the Laplacian in \eqref{ED1}. 
}
\end{remark}

To prove that the sequences $\{u^\eps, v^\eps\}_\eps$ converge we will apply an Arzel\`{a}-Ascoli type lemma.
To this end we need to show a sort of asymptotic continuity that is based on estimates
for both value functions $(u^\eps, v^\eps)$ near the boundary (these estimates can be extended
to the interior via a coupling probabilistic argument). 
In fact, to see an asymptotic continuity close to a boundary point, we are able to show that both players have strategies that 
enforce the game to end near a point $y\in \partial \Omega$ with 
high probability if we start close to that point no mater the strategy  chosen by the other player.  
This allows us to obtain a sort of asymptotic equicontinuity close to the boundary leading to
uniform convergence in the whole $\overline{\Omega}$. Note that, in general the value functions
$(u^\eps,v^\eps)$ are discontinuous in $\Omega$ (this is due to the fact that we make discrete steps)
and therefore showing uniform convergence to a continuos limit is a difficult task.

Let us see formally why a uniform limit $(u,v)$ is a solution to equation \eqref{ED1}.
By subtracting $u^\eps (x)$ and dividing by $\eps^2$ on both sides we get
\[
0=(v^{\eps}(x)-u^{\eps}(x))+(1-\eps^{2})
\left\{ \frac{ \half \sup_{y \in B_{\eps}(x)} u^{\eps}(y) 
+ \half \inf_{y \in B_{\eps}(x)}u^{\eps}(y) -u^{\eps}(x) }{\eps^{2}} \right\}.
\]
which in the limit approximates the first equation in our system \eqref{ED1}
(the terms into brackets approximate
 the second derivative of $u$ in the direction of its gradient).
 Similarly, the second equation in the DPP can be written as
 \[
 0 = (u^{\eps}(x)-v^{\eps}(x))+(1-\eps^{2})\kint_{B_{\eps}(x)} \frac{(v^{\eps}(y)- v^{\eps}(x))}{\eps^2} dy
 \]
 that approximates solutions to the second equation in \eqref{ED1}
.

\medskip

The paper is organized as follows: in Section \ref{sect-prelim} we include 
the precise definition of what we will understand by a viscosity solution for our system and we state a key preliminary result
from probability theory
(the Optional Stopping Theorem);
Section \ref{sect-game2} contains a detailed description of the game, also in Section \ref{sect-game2} we show that there is
a value of the game that satisfies the DPP \eqref{DPP} and prove uniqueness for the DPP (we prove Theorem \ref{teo.dpp2}); 
next, in Section \ref{sect-convergence} 
we analyze the game and show that value functions converge uniformly along subsequences to a pair of continuous functions; 
in Section \ref{sect-limiteviscoso} we prove that the limit is a viscosity solution to our system (up to this point we obtain the first
part of Theorem \ref{teo.converge2}) and in Section \ref{sect-uniqueness} we show uniqueness
of viscosity solutions to the system, ending the proof of Theorem \ref{teo.converge2}.
Finally, in Section \ref{sect.extensiones} we collect some comments on possible extensions of our results.

\section{Preliminaries.} \label{sect-prelim}

In this section we include the precise definition of what we understand as a viscosity 
solution for the system \eqref{ED1} and we include the precise statement of the 
Optional Stopping Theorem
that will be needed when dealing with the probabilistic part of our arguments. 

\subsection{Viscosity solutions}
We begin by stating the definition of a viscosity solution to a fully nonlinear second order elliptic PDE.
We refer to
\cite{CIL} for general results on viscosity solutions.
Fix a function
\[
P:\Omega\times\R\times\R^N\times\mathbb{S}^N\to\R
\]
where $\mathbb{S}^N$ denotes the set of symmetric $N\times N$ matrices.
We want to consider the PDE 
\begin{equation}
\label{eqvissol}
P(x,u (x), Du (x), D^2u (x)) =0, \qquad x \in \Omega.
\end{equation}

The idea behind Viscosity Solutions is to use the maximum principle in order to 
``pass derivatives to smooth test functions''. This idea allows us to consider operators in non divergence form.
We will assume that $P$ is degenerate elliptic, that is, $P$ satisfies a monotonicity property with respect
to the matrix variable, that is,
\[
X\leq Y \text{ in } \mathbb{S}^N \implies P(x,r,p,X)\geq P(x,r,p,Y)
\]
for all $(x,r,p)\in \Omega\times\R\times\R^N$.

Here
we have an equation that involves the $\infty$-laplacian that is not well defined when the gradient vanishes. 
In order to be able to handle this issue, we need to consider the lower semicontinous, $P_*$, and upper semicontinous, $P^*$, envelopes of 
$P$.
These functions are given by
\[
\begin{array}{ll}
P^*(x,r,p,X)& \displaystyle =\limsup_{(y,s,w,Y)\to (x,r,p,X)}P(y,s,w,Y),\\
P_*(x,r,p,X)& \displaystyle =\liminf_{(y,s,w,Y)\to (x,r,p,X)}P(y,s,w,Y).
\end{array}
\]
These functions coincide with $P$ at every point of continuity of $P$ and are lower and upper semicontinous respectively.
With these concepts at hand we are ready to state the definition of a viscosity solution to
\eqref{eqvissol}.

\begin{definition}
\label{def.sol.viscosa.1}
A lower semi-continuous function $ u $ is a viscosity
supersolution of \eqref{eqvissol} if for every $ \phi \in C^2$ such that $ \phi $
touches $u$ at $x \in \Omega$ strictly from below (that is, $u-\phi$ has a strict minimum at $x$ with $u(x) = \phi(x)$), we have
$$P^*(x,\phi(x),D \phi(x),D^2\phi(x))\geq 0.$$

An upper semi-continuous function $u$ is a subsolution of \eqref{eqvissol} if
for every $ \psi \in C^2$ such that $\psi$ touches $u$ at $ x \in
\Omega$ strictly from above (that is, $u-\psi$ has a strict maximum at $x$ with $u(x) = \psi(x)$), we have
$$P_*(x,\phi(x),D \phi(x),D^2\phi(x))\leq 0.$$

Finally, $u$ is a viscosity solution of \eqref{eqvissol} if it is both a suoer- and a subsolution.
\end{definition}

In our system \eqref{ED1} we have two equations given by the functions
$$
F_1 (x,u,p,X)  =  - \frac12 \langle X\frac{p}{|p|} , \frac{p}{|p|} \rangle + u - v(x)=0
$$
and 
$$
F_2  (x,v,q,Y) = - \frac{\kappa}{2} trace (Y) + v - u(x)=0.
$$

Then, the definition of a viscosity solution for the system \eqref{ED1} that we will use here is the following.

\begin{definition}
\label{def.sol.viscosa.system}
A pair of continuous functions $ u, v :\overline{\Omega} \mapsto \mathbb{R} $ is a viscosity
solution of \eqref{ED1} if 
$$
u|_{\partial \Omega} = f, \qquad v|_{\partial \Omega} = g,
$$
$$
\mbox{$u$ is a viscosity solution to 
$F_1 (x,u,D u, D^2u) = 0$}
$$
$$
and 
$$
$$
\mbox{$v$ is a viscosity solution to 
$
F_2 (x,v, D v, D^2v) = 0$}
$$
in the sense of Definition \ref{def.sol.viscosa.1}.
\end{definition}

\begin{remark} {\rm
We remark that, according to our definition, in the equation for $u$, 
as the other component $v$ is continuous, we have that $F_1$ depends on $x$ via $v(x)$ 
(and similarly for $F_2$ that depend on $x$ as $u(x)$). 
That is, we understand a solution to \eqref{ED1} as a pair of continuous up to the boundary functions
that satisfies the boundary conditions pointwise and such that $u$ is a viscosity solution to 
the first equation in the system in the viscosity sense (with $v$ as a fixed continuos function of $x$ in $F_1$) and $v$ solves the second equation 
in the system (regarding $u$ as a fixed function of $x$ in $F_2$).

Also notice that we have that both $u$ and $v$ are assumed to be continuous in $\overline{\Omega}$ and then the boundary data
$f$ and $g$
are taken on $\partial \Omega$ with continuity. 
}
\end{remark}

\subsection{Probability. The Optional Stopping Theorem.}
We briefly recall (see \cite{Williams}) that a sequence of random variables
$\{M_{k}\}_{k\geq 1}$ is a supermartingale (submartingales) if
$$ \E[M_{k+1}\arrowvert M_{0},M_{1},...,M_{k}]\leq M_{k} \ \ (\geq)$$
Then, the Optional Stopping Theorem, that we will call {\it (OSTh)} in what follows, says:
given $\tau$ a stopping time such that one of the following conditions hold,
\begin{itemize}
\item[(a)] The stopping time $\tau$ is bounded a.s.;
\item[(b)] It holds that $\E[\tau]<\infty$ and there exists a constant $c>0$ such that $$\E[M_{k+1}-M_{k}\arrowvert M_{0},...,M_{k}]\leq c;$$
\item[(c)] There exists a constant $c>0$ such that $|M_{\min \{\tau,k\}}|\leq c$ a.s. for every $k$.
\end{itemize}
Then 
$$ \E[M_{\tau}]\leq \E [M_{0}] \ \ (\geq)$$
if $\{M_{k}\}_{k\geq 0}$ is a supermartingale (submartingale).
For the proof of this classical result we refer to \cite{Doob,Williams}. 

\section{A two-player game} \label{sect-game2}

In this section, we describe in detail the two-player zero-sum game 
presented in the introduction.
Let $\Omega \subset\R^N$ be a bounded smooth domain and fix $\eps>0$. 
The game takes place in two boards (that we will call board 1 and board 2), 
that are two copies of $\mathbb{R}^N$
with the same domain $\Omega$ inside. Fix two Lipschitz functions 
$\ol{f}:\R^{n} \backslash \Om \rightarrow \R$ and $\ol{g}:\R^{N} \backslash \Om \rightarrow \R$ 
that are going to give the final payoff of the game 
when we exit $\Omega$ in board 1 and 2 respectively.

A token is placed at $x_0\in\Omega $ in one of the two boards. 
When we play in the first board, with probability $1-\eps^2$ we play \textit{Tug-of-War}, 
the game introduced in \cite{PSSW}, a fair coin (with probability $\frac12$ of heads and tails) is tossed 
and the player who win the coin toss chooses the next position of the game
inside the ball $B_\eps (x_0)$ in the first board. With probability $\eps^2$ we jump to the other
board, the next position of the toke in $x_0$ but now in board 2. 
If $x_0$ is in the second board then with probability $1-\eps^2$ the new position of the game is 
chosen at random in the ball $B_\eps (x_0)$ (with uniform probability) and with probability $\eps^2$
the position jumps to the same $x_0$ but in the first board.
The position of the token will be denoted by $(x,j)$ where $x \in \mathbb{R}^N$ and $j=1,2$ ($j$ encodes
the boars in which the token is at position $x$). 
Then, after one movement, the players continue playing with the same
rules from the new position of the token $x_1$ in its corresponding board, 1 or 2.
The game ends when 
the position of the token leaves the domain $\Omega$. That is, let $\tau$ be the stopping time given
by the first time at which $x_{\tau} \in \R^{N} \backslash \Omega$. 
If $x_{\tau}$ is in the first board then Player I gets $\ol{f}(x_{\tau})$ (and Player II pays that quantity), while
in the token leaves $\Omega$ in the second board 
Player I gets $\ol{g}(x_{\tau})$ (and Player II pays that amount). 
We have that the game generates a sequence of states
$$
P=\{ (x_{0},j_{0}),(x_{1},j_{1}),...,(x_{\tau},j_{\tau})\}
$$
with $j_{i}\in\{1,2\}$ and $x_{i}$ in the board $j_{i}$. 
The dependence of the position of the token 
in one of the boards, $j_i$, will be made explicit only when 
needed.

A strategy $S_\I$ for Player~I is a function defined on the
partial histories that gives the next position of the game provided Player $\I$ wins the coin toss
(and the token is and stays in the first board)
\[
S_\I{\left((x_0,j_{0}),(x_1,,j_{1}),\ldots,(x_n,,1)\right)}= (x_{n+1},1) \qquad \mbox{with } x_{n+1} \in B_\eps (x_n).
\]
Analogously, a strategy $S_\II$ for Player~II is a function defined on the
partial histories that gives the next position of the game provided Player $\II$ is who wins the coin toss
(and the token stays at the first board).

When the two players fix their strategies $S_I$ and $S_{II}$ we can compute the expected outcome as follows:
Given the sequence $x_0,\ldots,x_n$ with $x_k\in\Om$, if $x_k$ belongs to the first board, the next game position is distributed according to
the probability
\[
\pi_{S_\I,S_\II,1}((x_0,j_0),\ldots,(x_k,1),{A}, B)= \frac{1-\eps^2}{2} \delta_{S_\I ((x_0,j_0),\ldots,(x_k,1))} (A) +
\frac{1-\eps^2}{2} \delta_{S_\II ((x_0,j_0),\ldots,(x_k,1))} (A) + \eps^2 \delta_{x_k} (B).
\]
Here $A$ is a subset in the first board while $B$ is a subset in the second board. 
If 
$x_k$ belongs to the second board, the next game position is distributed according to
the probability
\[
\pi_{S_\I,S_\II,2}((x_0,j_0),\ldots,(x_k,2),{A}, B)= (1-\eps^2) U(B_\eps(x_k)) (B) +
\eps^2 \delta_{x_k} (A).
\]

By using the Kolmogorov's extension theorem and the one step transition probabilities, we can build a
probability measure $\mathbb{P}^{x_0}_{S_\I,S_\II}$ on the
game sequences (taking onto account the two boards). The expected payoff, when starting from $(x_0,j_0)$ and
using the strategies $S_\I,S_\II$, is
\begin{equation}
\label{eq:defi-expectation}
\mathbb{E}_{S_{\I},S_\II}^{(x_0,j_0)} [ h (x_\tau) ]=\int_{H^\infty} h (x_\tau)  \,  d
\mathbb{P}^{x_0}_{S_\I,S_\II}
\end{equation}
(here we use $h=f$ if $x_\tau$ is in the first board or $h=g$ if $x_\tau$ is in the second board).

The \emph{value of the game for Player I} is given by
\[
u^\eps_\I(x_0)=\inf_{S_\I}\sup_{S_{\II}}\,
\mathbb{E}_{S_{\I},S_\II}^{(x_0,1)}\left[h (x_\tau) \right]
\]
for $x_0\in\Omega$ in the first board ($j_0 =1$), and by 
\[
v^\eps_\I(x_0)=\inf_{S_\I}\sup_{S_{\II}}\,
\mathbb{E}_{S_{\I},S_\II}^{(x_0,2)}\left[h (x_\tau) \right]
\]
for $x_0\in\Omega$ in the second board ($j_0 =2$).

The \emph{value of the game for Player II} is given by the same formulas just reversing the $\inf$--$\sup$,
\[
u^\eps_\II(x_0)=\sup_{S_{\II}}\inf_{S_\I}\,
\mathbb{E}_{S_{\I},S_\II}^{(x_0,1)}\left[h (x_\tau) \right], 
\]
for $x_0$ in the first board and 
\[
v^\eps_\II(x_0)=\sup_{S_{\II}}\inf_{S_\I}\,
\mathbb{E}_{S_{\I},S_\II}^{(x_0,2)}\left[h (x_\tau) \right], 
\]
for $x_0$ in the second board.

Intuitively, the values $u_\I(x_0)$ and $u_\II(x_0)$ are the best
expected outcomes each player can guarantee when the game starts at
$x_0$ in the first board while $v_\I(x_0)$ and $v_\II(x_0)$ are the best
expected outcomes for each player in the second board.

If $u^\eps_\I= u^\eps_\II$ and $v^\eps_\I= v^\eps_\II$, we say that the game has a value.

Before proving that the game has a value,
let us observe that the game ends almost surely no matter the strategies used by the players, that is $\P (\tau =+\infty) =0$, and therefore
the expectation \eqref{eq:defi-expectation} is well defined. This fact is due to the random movements that we
make in the second board (that kicks us out of the domain in a finite number of plays without changing boards with positive probability). 

\begin{proposition} \label{prop.juego.termina}
We have that 
$$
\mathbb{P} \Big(\mbox{the game ends in a finite number of plays}\Big)=1.
$$
\end{proposition}

\begin{proof} Let us start by showing that the game ends in a finite number of plays if we start with the token in the second board. 
Let $\xi\in\R^{n}$ with $| \xi |=1$ be a fixed direction. Consider the set
$$
T_{\xi,x_{k}}=\Big\{ y\in\R^{n}:y\in B_{\eps}(x_{k})\wedge \langle y-(x_{k}+\frac{\eps}{2}\xi),\xi\rangle\geq 0 \Big\}
$$
that is a part of the ball where the points are at distance $\frac{\eps}{2}$ from the center and are in the same direction. 
Then, starting from any point in $\Om$ if in every play we choose a point in $T_{\xi,x_{k}}$ (without changing boards) 
in at most $\lceil \frac{4R}{\eps}\rceil$ steps we will be a out of $\Om$ (here $R=diam(\Om)$). 
As the set $T_{\xi,x_{k}}$ has positive measure it holds that
$$
\mathbb{P}(x_{k+1}\in T_{\xi,x_{k}}\arrowvert x_{k}):=\alpha > 0.
$$
Therefore, we have a positive probability of ending the game in less than $\lceil \frac{4R}{\eps}\rceil$ plays 
$$
\mathbb{P}(\mbox{the game ends in $\lceil \frac{4R}{\eps}\rceil$ plays})\geq [(1-\eps^{2})\alpha]^{K}=r>0.
$$
Hence
$$
\mathbb{P}(\mbox{the game continues after $\lceil \frac{4R}{\eps}\rceil$ plays})\leq 1-r,
$$
and then  
$$
\mathbb{P}(\mbox{the game does not end in a finite number of plays}) =0.
$$ 

Now, if we start in the first board the probability of not changing the board in $n$ plays is $(1-\eps^{2})^{n}$. Therefore,
we will change to the second board (or end the game) with probability one in a finite number of plays. Hence, 
we end the game or we are in the previous situation with probability one 

This implies that the game ends almost surely in a finite number of plays.
\end{proof}

To see that the game has a value, we first observe that we have existence
of $(u^\eps, v^\eps)$, a pair of functions that satisfies the DPP.
The existence of such a pair can be obtained by Perron's method. In fact, let us start considering the following set (that is composed by pairs of
functions 
that are sub solutions to our DPP). Let
\begin{equation} \label{kkk}
C=\max \Big\{ \|\ol{f}\|_\infty , \|\ol{g}\|_\infty
\Big\},
\end{equation}
and consider the set of functions 
\begin{equation}
\label{A}
{A} =\displaystyle \Big\{ (z^{\eps},w^{\eps}) : \mbox{ are bounded above by $C$ and verify (\textbf{e}) } \Big\},
\end{equation}
with
\begin{equation}
\label{e} \tag{\bf{e}}
\displaystyle \left\lbrace
\begin{array}{ll}
 \displaystyle z^{\eps}(x)\leq\eps^{2}w^{\eps}(x)+(1-\eps^{2})\Big\{\half \sup_{y \in B_{\eps}(x)}z^{\eps}(y) + \half \inf_{y \in B_{\eps}(x)}z^{\eps}(y)\Big\} \qquad &  \ x \in \Omega,  \\[10pt]
   \displaystyle w^{\eps}(x)\leq\eps^{2}z^{\eps}(x)+(1-\eps^{2})\kint_{B_{\eps}(x)}w^{\eps}(y)dy  \qquad &  \ x \in \Omega,  \\[10pt]
z^{\eps}(x) \leq \ol{f}(x) \qquad & \ x \in \R^{N} \backslash \Omega,  \\[10pt]
w^{\eps}(x) \leq \ol{g}(x) \qquad & \ x \in \R^{N} \backslash \Omega. 
\end{array}
\right.
\end{equation}

\begin{remark} {\rm Notice that we need to impose that $(z^{\eps},w^{\eps})$ are bounded since 
$$
z^{\eps} (x) =
\left\{ 
\begin{array}{ll}
+\infty \qquad & x \in \Omega \\[5pt]
\ol{f} \qquad & x \not\in \Omega
\end{array}
\right.
\qquad
\mbox{and} 
\qquad
w^{\eps} (x) =
\left\{ 
\begin{array}{ll}
+\infty \qquad & x \in \Omega \\[5pt]
\ol{g} \qquad & x \not\in \Omega
\end{array}
\right.
$$
satisfy \textbf{e}.}
\end{remark}

Observe that 
 ${A} \neq \emptyset$. To see this fact, we just take 
 $z^{\eps}=-C$ and $ w^{\eps}=-C$
 with $C$ given by \eqref{kkk}.
Now we let
 \begin{equation}
 \label{u}
 u^{\eps}(x)=\sup_{(z^{\eps},w^{\eps})\in {A}}z^{\eps}(x) 
 \qquad \mbox{and} \qquad
   v^{\eps}(x)=\sup_{(z^{\eps},w^{\eps})\in {A}}w^{\eps}(x) .
   \end{equation}
Our goal is to show that in this way we find a solution to the DPP. 

\begin{proposition} \label{prop.DPP.tiene.sol}
The pair $(u^{\eps},v^{\eps})$ given by \eqref{u} is a solution to the DPP \eqref{DPP}
\end{proposition}

\begin{proof}
First, let us see that $(u^{\eps},v^{\eps})$ belongs to the set ${A}$. To this end we first observe that 
$u^{\eps}$ y $v^{\eps}$ are bounded by $C$ and verify 
 $u^{\eps}(x)\leq \ol{f}(x)$ and $v^{\eps}(x)\leq \ol{g}(x)$ for $x\in\R^{N}\backslash\Om$. 
 Hence we need to check \eqref{e} for $x\in\Om$. Take $(z^{\eps},w^{\eps})\in {A}$ and fix $x\in\Om$. Then,
$$z^{\eps}(x)\leq\eps^{2}w^{\eps}(x)+(1-\eps^{2})\Big\{\half \sup_{y \in B_{\eps}(x)}z^{\eps}(y) + \half \inf_{y \in B_{\eps}(x)}z^{\eps}(y)\Big\}.$$
As $z^{\eps}\leq u^{\eps}$ and $w^{\eps}\leq v^{\eps}$ we obtain 
$$z^{\eps}(x)\leq\eps^{2}v^{\eps}(x)+(1-\eps^{2})\Big\{\half \sup_{y \in B_{\eps}(x)}u^{\eps}(y) + \half \inf_{y \in B_{\eps}(x)}u^{\eps}(y)\Big\}.$$
Taking supremum in the left hand sise we obtain
$$u^{\eps}(x)\leq\eps^{2}v^{\eps}(x)+(1-\eps^{2})\Big\{\half \sup_{y \in B_{\eps}(x)}u^{\eps}(y) + \half \inf_{y \in B_{\eps}(x)}u^{\eps}(y)\Big\}.$$
In an analogous way we obtain  
$$ v^{\eps}(x)\leq\eps^{2}u^{\eps}(x)+(1-\eps^{2})\kint_{B_{\eps}(x)}v^{\eps}(y)dy,$$
and we conclude that $(u^{\eps},v^{\eps}) \in {A}$.

To end the proof we need to see that $(u^{\eps},v^{\eps})$ verifies the equalities in the equations in condition \eqref{e}. 
We argue by contradiction and assume that there is a point $x_{0}\in\R^{n}$ where an inequality in \eqref{e} is strict.
First, assume that $x_{0}\in\R^{n}\backslash\Om$, and that we have $u^{\eps}(x_{0})<\ol{f}(x_{0})$.
Then, take $u^{\eps}_{0}$ defined by $u^{\eps}_{0}(x)=u^{\eps}(x)$ for $x\neq x_0$ and $u^{\eps}_{0}(x_{0})=\ol{f}(x_{0})$. 
The pair $(u^{\eps}_{0},v^{\eps})$ belongs to ${A}$ but $u^{\eps}_{0}(x_{0})>u^{\eps}(x_{0})$
which is a contradiction. We can argue is a similar way if $v^{\eps}(x_{0})<\ol{g}(x_{0})$. 
Next, we consider a point $x_{0}\in\Om$ with one of the inequalities in \textbf{e} strict. Assume that 
$$u^{\eps}(x_{0})<\eps^{2}v^{\eps}(x_{0})+(1-\eps^{2})\Big\{\half \sup_{y \in B_{\eps}(x_{0})}u^{\eps}(y) + \half \inf_{y \in B_{\eps}(x_{0})}u^{\eps}(y)\Big\}.$$
Let
$$ \delta=\eps^{2}v^{\eps}(x_{0})+(1-\eps^{2})\Big\{\half \sup_{y \in B_{\eps}(x_{0})}u^{\eps}(y) + \half \inf_{y \in B_{\eps}(x_{0})}u^{\eps}(y)\Big\}-u^{\eps}(x_{0})>0, $$
and consider the function $u^{\eps}_{0}$ given by;
$$
 u^{\eps}_{0} (x) =\left\lbrace
 \begin{array}{ll}
 u^{\eps}(x) &  \ \ x \neq x_{0},  \\[5pt]
 \displaystyle u^{\eps}(x)+\frac{\delta}{2} &  \ \ x =x_{0} . \\
 \end{array}
 \right.
$$
Observe that
$$u^{\eps}_{0}(x_{0})=u^{\eps}(x_{0})+\frac{\delta}{2}<\eps^{2}v^{\eps}(x_{0})+(1-\eps^{2})\Big\{\half \sup_{y \in B_{\eps}(x_{0})}u^{\eps}(y) + \half \inf_{y \in B_{\eps}(x_{0})}u^{\eps}(y)\Big\}$$
and hence
$$u^{\eps}_{0}(x_{0})<\eps^{2}v^{\eps}(x_{0})+(1-\eps^{2})\Big\{\half \sup_{y \in B_{\eps}(x_{0})}u^{\eps}_{0}(y) + \half \inf_{y \in B_{\eps}(x_{0})}u^{\eps}_{0}(y)\Big\}.$$
Then we have that $(u^{\eps}_{0},v^{\eps})\in {A}$ but $u^{\eps}_{0}(x_{0})>u^{\eps}(x_{0})$ reaching again a contradiction.

In an analogous way we can show that when 
$$v^{\eps}(x_{0})<\eps^{2}u^{\eps}(x_{0})+(1-\eps^{2})\kint_{B_{\eps}(x_{0})}v^{\eps}(y)dy,$$
we also reach a contradiction.
\end{proof}

Now, concerning the value functions of our game, we know that $u^\eps_\I\geq u^\eps_\II$ 
and $v^\eps_\I\geq v^\eps_\II$
(this is immediate from the definitions). Hence, 
to obtain uniqueness of solutions of the DPP and existence of value functions for our game, 
it is enough to show that $u^\eps_\II \geq u^\eps\geq u^\eps_\I$ and 
$v^\eps_\II \geq v^\eps\geq v^\eps_\I$. To show this result
we will use the \textit{OSTh} for sub/supermartingales
(see Section \ref{sect-prelim}).

\begin{theorem}
Gigen $\eps>0$ let $(u^{\eps},v^{\eps})$ a pair of functions that verifies the DPP \eqref{DPP}, then it holds that
$$
u^{\eps}(x_{0})=\sup_{S_{I}}\inf_{S_{II}}\E^{(x_{0},1)}_{S_{I},S_{II}}[h( x_{\tau})] 
$$
if $x_{0} \in \Om$ is in the first board and
$$
v^{\eps}(x_{0})=\sup_{S_{I}}\inf_{S_{II}}\E^{(x_{0},2)}_{S_{I},S_{II}}[h( x_{\tau})]
$$
 if $x_{0} \in \Om$ is in the second board. 
 
 Moreover, we can interchange $\inf$ with $\sup$ in the previous identities, that is, the game has a value. 
 This value can be characterized as the unique solution to the DPP.  
\end{theorem}

\begin{proof}
Given $\eps>0$ we have proved the existence of a solution to the DPP $(u^{\eps},v^{\eps})$. Fix $\delta>0$. 
Assume that we start with $(x_{0},1)$, that is, the initial position is at board 1. We choose a strategy for 
Player I as follows:
$$
x_{k+1}^{I}=S_{I}^{*}(x_{0},...,x_{k}) 
\qquad 
\mbox{is such that} \qquad
\sup_{y\in B_{\eps}(x_{k})}u^{\eps}(y) - \frac{\delta}{2^{k}}\leq u^{\eps}(x_{k+1}^{I}).
$$
Given this strategy for Player I and any strategy $S_{II}$ for Player II we consider the sequence of random variables
given by
$$
M_{k}=\left\lbrace
 \begin{array}{ll}
 \displaystyle u^{\eps}(x_{k})-\frac{\delta}{2^{k}} & \ \ {if } \ (j_{k}=1),  \\[10pt]
 \displaystyle v^{\eps}(x_{k})-\frac{\delta}{2^{k}} & \ \ {if } \ (j_{k} = 2).
 \end{array}
 \right.
$$
Let us see that $(M_{k})_{\kappa\geq 0}$ is a submartingale. 
To this end we need to estimate 
$$
\E^{(x_{0},1)}_{S_{I}^{*},S_{II}}[M_{k+1}\arrowvert M_{k}]=\E^{(x_{0},1)}_{S_{I}^{*},S_{II}}[M_{k+1}\arrowvert (x_{k},j_{k})].
$$
We consider two cases:

\textbf{Case 1:} Assume that $j_{k}=1$, then 
$$
\begin{array}{l}
\displaystyle 
\E^{(x_{0},1)}_{S_{I}^{*},S_{II}}[M_{k+1}\arrowvert (x_{k},1)] 
 \displaystyle =(1-\eps^{2})\E^{(x_{0},1)}_{S_{I}^{*},S_{II}}[M_{k+1}\arrowvert (x_{k},1)\wedge j_{k+1}=1]+\eps^{2}\E^{(x_{0},1)}_{S_{I}^{*},S_{II}}[M_{k+1}\arrowvert (x_{k},1)\wedge j_{k+1}=2].
\end{array}
$$
Her we used that the probability of staying in the same board is $(1-\eps^{2})$ and the probability of jumping to the other board is
$\eps^{2}$. Now, if $j_{k}=1$ and $j_{k+1}=2$ then $x_{k+1}=x_{k}$ (we just changed boards). 
On the other hand, if we stay in the first board we obtain
$$
\E^{(x_{0},1)}_{S_{I}^{*},S_{II}}[M_{k+1}\arrowvert (x_{k},1)]=(1-\eps^{2})\Big\{\half u^{\eps}(x_{k+1}^{I})+\half u^{\eps}(x_{k+1}^{II})-\frac{\delta}{2^{k+1}}\Big\}+\eps^{2}(v^{\eps}(x_{k})-\frac{\delta}{2^{k+1}}).
$$
Since we are using the strategies $S_{I}^{*}$ and $S_{II}$, it holds that 
$$
\sup_{y\in B_{\eps}(x_{k})}u^{\eps}(y) - \frac{\delta}{2^{k}}\leq u^{\eps}(x_{k+1}^{I})
\qquad
\mbox{and}\qquad
\inf_{y\in B_{\eps}(x_{k})}u^{\eps}(y)\leq u^{\eps}(x_{k+1}^{II}).
$$
Therefore, we arrive to
$$
\E^{(x_{0},1)}_{S_{I}^{*},S_{II}}[M_{k+1}\arrowvert (x_{k},1)]\geq(1-\eps^{2})\Big\{\half (\sup_{y\in B_{\eps}(x_{k})}u^{\eps}(y) - \frac{\delta}{2^{k}})+\half \inf_{y\in B_{\eps}(x_{k})}u^{\eps}(y)\Big\}+\eps^{2}v^{\eps}(x_{k})-\frac{\delta}{2^{k+1}},
$$
that is,
$$
\begin{array}{l}
\displaystyle 
\E^{(x_{0},1)}_{S_{I}^{*},S_{II}}[M_{k+1}\arrowvert (x_{k},1)] 
 \displaystyle \geq(1-\eps^{2})\Big\{\half \sup_{y\in B_{\eps}(x_{k})}u^{\eps}(y)+\half \inf_{y\in B_{\eps}(x_{k})}u^{\eps}(y)\Big\}+\eps^{2}v^{\eps}(x_{k})-(1-\eps^{2}) \frac{\delta}{2^{k+1}}-\frac{\delta}{2^{k+1}}.
\end{array}
$$
As $u^{\eps}$ is a solution to the DPP \eqref{DPP} we obtain 
$$
\E^{(x_{0},1)}_{S_{I}^{*},S_{II}}[M_{k+1}\arrowvert (x_{k},1)]\geq u^{\eps}(x_{k})-\frac{\delta}{2^{k}}=M_{k}
$$
as we wanted to show.

\textbf{Case 2:} Assume that $j_{k}=2$. With the same ideas used before we get
$$
\begin{array}{l}
\displaystyle 
\E^{(x_{0},1)}_{S_{I}^{*},S_{II}}[M_{k+1}\arrowvert (x_{k},2)] 
\displaystyle =(1-\eps^{2})\E^{(x_{0},1)}_{S_{I}^{*},S_{II}}[M_{k+1}\arrowvert (x_{k},2)\wedge j_{k+1}=2]+\eps^{2}\E^{(x_{0},1)}_{S_{I}^{*},S_{II}}[M_{k+1}\arrowvert (x_{k},2)\wedge j_{k+1}=1].
\end{array}
$$
Remark that when $j_{k}=j_{k+1}=2$ (this means that we play in the second board) with $x_{k}\in\Om$, then $x_{k+1}$ is chosen with
uniform probability in the ball $B_{\eps}(x_{k})$. Hence,
$$
\begin{array}{l}
\displaystyle \E^{(x_{0},1)}_{S_{I}^{*},S_{II}}[M_{k+1}\arrowvert (x_{k},2)\wedge j_{k+1}=2]\ \displaystyle =\E^{(x_{0},1)}_{S_{I}^{*},S_{II}}[v^{\eps}(x_{k+1})-\frac{\delta}{2^{k+1}}\arrowvert (x_{k},2)\wedge j_{k+1}=2]
\displaystyle =\kint_{B_{\eps}(x_{k})}v^{\eps}(y)dy-\frac{\delta}{2^{k+1}}.
\end{array}
$$
On the other hand, 
$$
\E^{(x_{0},1)}_{S_{I}^{*},S_{II}}[M_{k+1}\arrowvert (x_{k},2)\wedge j_{k+1}=1]=u^{\eps}(x_{k})-\frac{\delta}{2^{k+1}}.
$$
Collecting these estimates we obtain
$$
\begin{array}{l}
\displaystyle 
\E^{(x_{0},1)}_{S_{I}^{*},S_{II}}[M_{k+1}\arrowvert (x_{k},2)]=(1-\eps^{2})\left(\kint_{B_{\eps}(x_{k})}v^{\eps}(y)dy-\frac{\delta}{2^{k+1}}
\right)+\eps^{2}(u^{\eps}(x_{k})-\frac{\delta}{2^{k+1}}) \\[10pt]
\qquad \displaystyle \geq(1-\eps^{2})\kint_{B_{\eps}(x_{k})}v^{\eps}(y)dy+\eps^{2}u^{\eps}(x_{k})-\frac{\delta}{2^{k}},
\end{array}
$$
that is,
$$
\E^{(x_{0},1)}_{S_{I}^{*},S_{II}}[M_{k+1}\arrowvert (x_{k},2)]\geq v^{\eps}(x_{k})-\frac{\delta}{2^{k}}=M_{k}.
$$
Here we used that $v^{\eps}$ is a solution to the DPP, \eqref{DPP}. This ends the second case.

Therefore $(M_{k})_{k\geq 0}$ is a \textit{submartingale}. Using the \textit{OSTh} 
(recall that we have proved that $\tau$ is finite a.s. and that we have that $M_k$ is uniformly bounded) we conclude that 
$$
\E^{(x_{0},1)}_{S_{I}^{*},S_{II}}[M_{\tau}]\geq M_{0}
$$
where $\tau$ is the first time such that $x_{\tau}\notin\Om$ in any of the two boards. Then, 
$$
\E^{(x_{0},1)}_{S_{I}^{*},S_{II}}[\mbox{final payoff}]\geq u^\eps (x_{0})-\delta.
$$
We can compute the infimum  in $S_{II}$ and then
the supremum in $S_{I}$ to obtain
$$
\sup_{S_{I}}\inf_{S_{II}}\E^{(x_{0},1)}_{S_{I},S_{II}}[\mbox{final payoff}]\geq u^{\eps}(x_{0})-\delta.
$$

We just observe that if we have started in the second board the previous computations show that
$$
\sup_{S_{I}}\inf_{S_{II}}\E^{(x_{0},2)}_{S_{I},S_{II}}[\mbox{final payoff}]\geq v^{\eps}(x_{0})-\delta.
$$

Now our goal is to prove the reverse inequality (interchanging inf and sup). 
To this end we define an strategy for Player II with
$$
x_{k+1}^{II}=S_{II}^{*}(x_{0},...,x_{k})
\qquad
\mbox{is such that} 
\qquad
\inf_{B_{\eps}(x_{k})}u^{\eps}(x_{k+1}^{II})+\frac{\delta}{2^{k}}\geq u^{\eps}(x_{k+1}^{II}),
$$
and consider the sequence of random variables 
$$
N_{k}=\left\lbrace
 \begin{array}{ll}
 \displaystyle u^{\eps}(x_{k})+\frac{\delta}{2^{k}} & \ \ \mbox{if }  j_{k}=1  \\[10pt]
 \displaystyle v^{\eps}(x_{k})+\frac{\delta}{2^{k}} & \ \ \mbox{if }   j_{k}=2. 
 \end{array}
 \right.
$$
Arguing as before we obtain that this sequence is a \textit{supermartingale}.
From the \textit{OSTh} we get 
$$
\E^{(x_{0},1)}_{S_{I}^{*},S_{II}}[N_{\tau}]\leq N_{0}
$$
where $\tau$ is the stopping time for the game. Then, 
$$
\E^{(x_{0},1)}_{S_{I},S_{II}^{*}}[\mbox{final payoff}]\leq u^\eps (x_{0})+\delta.
$$
Taking supremum in $S_{I}$ and then infimum in $S_{II}$ we obtain
$$
\inf_{S_{II}}\sup_{S_{I}}\E^{(x_{0},1)}_{S_{I},S_{II}}[\mbox{final payoff}]\leq u^{\eps}(x_{0})+\delta.
$$
As before, the same ideas starting at $(x_{0},2)$ give us
$$
\inf_{S_{II}}\sup_{S_{I}}\E^{(x_{0},1)}_{S_{I},S_{II}}[\mbox{final payoff}]\leq v^{\eps}(x_{0})+\delta.
$$

To end the proof we just observe that 
$$
\sup_{S_{I}}\inf_{S_{II}}\E_{S_{I},S_{II}}[\mbox{final payoff}]\leq \inf_{S_{II}}\sup_{S_{I}}\E_{S_{I},S_{II}}[\mbox{final payoff}].
$$
Therefore,
$$
u^{\eps}(x_{0})-\delta\leq\sup_{S_{I}}\inf_{S_{II}}\E_{S_{I},S_{II}}^{(x_{0},1)}[\mbox{final payoff}]
\leq \inf_{S_{II}}\sup_{S_{I}}\E_{S_{I},S_{II}}^{(x_{0},1)}[\mbox{final payoff}]\leq u^{\eps}(x_{0})+\delta
$$
and
$$
v^{\eps}(x_{0})-\delta\leq\sup_{S_{I}}\inf_{S_{II}}\E_{S_{I},S_{II}}^{(x_{0},2)}[\mbox{final payoff}]\leq 
\inf_{S_{II}}\sup_{S_{I}}\E_{S_{I},S_{II}}^{(x_{0},2)}[\mbox{final payoff}]\leq v^{\eps}(x_{0})+\delta.
$$
As $\delta >0$ is arbitrary the proof is finished.
\end{proof}

\begin{remark} {\rm One can obtain existence for the DPP considering, 
$$
\displaystyle (\textbf{e*}) : \left\lbrace
\begin{array}{ll}
 \displaystyle z^{\eps}(x)\geq\eps^{2}w^{\eps}(x)+(1-\eps^{2})\Big\{\half \sup_{y \in B_{\eps}(x)}z^{\eps}(y) + \half \inf_{y \in B_{\eps}(x)}z^{\eps}(y)\Big\} \qquad &  \ x \in \Omega,  \\[10pt]
   \displaystyle w^{\eps}(x)\geq\eps^{2}z^{\eps}(x)+(1-\eps^{2})\kint_{B_{\eps}(x)}w^{\eps}(y)dy  \qquad & \ x \in \Omega,  \\[10pt]
z^{\eps}(x) \geq \ol{f}(x) \qquad & \ x \in \R^{n} \backslash \Omega,  \\[10pt]
w^{\eps}(x) \geq \ol{g}(x) \qquad &  \ x \in \R^{n} \backslash \Omega. 
\end{array}
\right.
$$
and the associated set of functions
\begin{equation}
{B} =\displaystyle \Big\{ (z^{\eps},w^{\eps}) / \mbox{ are bounded functions such that} \ (\textbf{e*}) \Big\}.
\end{equation}
Now, we compute infimums, 
 \begin{equation}
 \label{u2}
 u^{\eps,*}(x)=\inf_{(z^{\eps},w^{\eps})\in {B}}z^{\eps}(x) 
\qquad \mbox{ and } \qquad
   v^{\eps,*}(x)=\inf_{(z^{\eps},w^{\eps})\in {B}}w^{\eps}(x),
   \end{equation}
 that are solutions to the DPP (this fact can be proved as we did for supremums).
Then, by the uniqueness to solutions to the DPP we have 
 $$
 u^{\eps,*} = u^{\eps} \qquad \mbox{and} \qquad v^{\eps,*} = v^{\eps}.
 $$
}
\end{remark}

\section{Uniform Convergence} \label{sect-convergence}

Now our aim is to pass to the limit in the values of the game
$$
u^\eps \to u,  \ v^\eps \to v \qquad \mbox{as } \eps \to 0
$$
and then in the next section to obtain that this limit pair $(u,v)$ is a viscosity solution to our system \eqref{ED1}.

To obtain a convergent subsequence $u^\eps \to u$ we will use the following
Arzela-Ascoli type lemma. For its proof see Lemma~4.2 from \cite{MPRb}.

\begin{lemma}\label{lem.ascoli.arzela} Let $\{u^\eps : \overline{\Omega}
\to \R,\ \eps>0\}$ be a set of functions such that
\begin{enumerate}
\item there exists $C>0$ such that $\abs{u^\eps (x)}<C$ for
    every $\eps >0$ and every $x \in \overline{\Omega}$,
\item \label{cond:2} given $\delta >0$ there are constants
    $r_0$ and $\eps_0$ such that for every $\eps < \eps_0$
    and any $x, y \in \overline{\Omega}$ with $|x - y | < r_0 $
    it holds
$$
|u^\eps (x) - u^\eps (y)| < \delta.
$$
\end{enumerate}
Then, there exists  a uniformly continuous function $u:
\overline{\Omega} \to \R$ and a subsequence still denoted by
$\{u^\eps \}$ such that
\[
\begin{split}
u^{\eps}\to u \qquad\textrm{ uniformly in }\overline{\Omega},
\mbox{ as $\eps\to 0$.}
\end{split}
\]
\end{lemma}

So our task now is to show that $u^\eps$ and $v^\eps$ both satisfy the hypotheses of the previous lemma. First, we
observe that they are uniformly bounded.

\begin{lemma}\label{lem.ascoli.arzela.acot} 
There exists a constant $C>0$ independent of $\eps$ such that $$\abs{u^\eps (x)}\leq C, \qquad \abs{v^\eps (x)}\leq C,$$ for
    every $\eps >0$ and every $x \in \overline{\Omega}$.
\end{lemma}

\begin{proof} 
It follows form our proof of existence of a solution to the DPP. In fact, we can take
$$
C = \max \{ \| g\|_\infty, \| f \|_\infty \},
$$
since the final payoff in any of the boards is bounded by this $C$.
\end{proof}

To prove the second hypothesis of Lemma \ref{lem.ascoli.arzela} we will need 
some key estimates according to the board in which we are playing.

\subsection{Estimates for the Tug-of-War game}

In this case we are going to assume that we are playing in board 1 (with the Tug-of-War game) all the time (without changing boards).
\begin{lemma} \label{lema.estim.ToW}
Given $\eta>0$ and $a>0$, there exist $r_{0}>0$ and $\eps_{0}>0$ such that, given $y\in\partial\Om$ and $x_{0}\in\Om$ with $| x_{0}-y |<r_{0}$, 
any of the two players has a strategy $S^{*}$ with which we obtain 
$$\mathbb{P}\Big(x_{\tau} : | x_{\tau}-y | < a \Big) \geq 1 - \eta \qquad \mbox{and} \qquad
\mathbb{P}\Big(\tau \geq \frac{a}{\eps^2}\Big)< \eta$$
  for $\eps<\eps_{0}$ and $x_{\tau}\in\R^{N}\backslash\Om$ the first position outside $\Om$.
\end{lemma}

This Lemma says that if we start playing close enough to $y\in\partial\Om$ we will finish quickly  (in a number of 
steps less than a small constant times $\eps^2$) and at a final position close to $y\in\partial\Om$ with high probability.

\begin{proof}
We can assume without loss of generality that $y=0\in\partial\Om$. In this case we will define the strategy $S^{*}$ (this strategy can be used
by any of the two players) \lq\lq point to the point $y=0$'' as follows
$$x_{k+1}=S^{*}(x_{0},x_{1},...,x_{k})=x_{k}+(\frac{\eps^{3}}{2^{k}}-\eps)\frac{x_{k}}{ |x_{k} |}$$
Now let us consider the random variables
$$ N_{k}= | x_{k}|+\frac{\eps^{3}}{2^{k}} $$
for $k\geq 0$ and play assuming that one of the players uses the $S^{*}$ strategy. The goal is to prove that $\{N_{k}\}_{k\geq 0}$ 
is \textit{supermartingale}, i.e.,
$$ \E[N_{k+1}\arrowvert N_{k}]\leq N_{k}. $$
Note that with probability 1/2 we obtain
$$x_{k+1}=x_{k} + (\frac{\eps^{3}}{2^{k}}-\eps)\frac{x_{k}}{ | x_{k} |}$$
this is the case when the player who uses the $S^{*}$ strategy wins the coin toss. On the other hand, we have 
$$| x_{k+1}|\leq | x_{k}| + \eps, $$ 
when the other player wins.
Then, we obtain
$$ \E\Big[| x_{k+1}| \arrowvert  x_{k} \Big]\leq \half \Big(| x_{k}| +(\frac{\eps^{3}}{2^{k}}-\eps)\Big) + \half(| x_{k} | + \eps)= | x_{k} |
+\frac{\eps^{3}}{2^{k+1}}.$$
Hence, we get
$$ \E \Big[N_{k+1}\arrowvert N_{k}\Big]=\E \Big[ | x_{k+1}| + \frac{\eps^{3}}{2^{k+1}}\arrowvert | x_{k} | + \frac{\eps^{3}}{2^{k}}\Big]
\leq |x_{k}| + \frac{\eps^{3}}{2^{k+1}}+\frac{\eps^{3}}{2^{k+1}}=N_{k}.$$
We just proved that $\{N_{k}\}_{k\geq 0}$ is a \textit{supermartingale}.
Now, let us consider the random variables
$$ (N_{k+1}-N_{k})^{2},$$
and the event
\begin{equation}
\label{Fk}
F_{k}=\{  the \ player \ who \ points \ to \ 0\in\partial\Om \ wins \ the \ coin \ toss  \}.
\end{equation}
Then we have the following
$$
\begin{array}{l}
\displaystyle \E[(N_{k+1}-N_{k})^{2}\arrowvert N_{k}]
 \displaystyle =\half\E[(N_{k+1}-N_{k})^{2}\arrowvert N_{k} \wedge F_{k}]+\half\E[(N_{k+1}-N_{k})^{2}\arrowvert N_{k} \wedge F_{k}^{c}]
\\[10pt]
\qquad \displaystyle \geq 
 \displaystyle  \half\E[(N_{k+1}-N_{k})^{2}\arrowvert N_{k} \wedge F_{k}].
\end{array}
$$
Let us observe that
$$
\begin{array}{l}
\displaystyle
\half\E[(N_{k+1}-N_{k})^{2}\arrowvert N_{k} \wedge F_{k}]
\displaystyle =\half\E \Big[( |x_{k}| -\eps+\frac{\eps^{3}}{2^{k}}+\frac{\eps^{3}}{2^{k+1}}- | x_{k}| -\frac{\eps^{3}}{2^{k}})^{2} \Big]
 \displaystyle =\half\E \Big[(-\eps+\frac{\eps^{3}}{2^{k+1}})^{2} \Big]\geq\frac{\eps^{2}}{3}
\end{array}
$$
if $\eps<\eps_{0}$ for $\eps_{0}$ small enough. With this estimate in mind we obtain
\begin{equation}
\label{e/3b}
\E[(N_{k+1}-N_{k})^{2}\arrowvert N_{k}] \geq \frac{\eps^{2}}{3}.
\end{equation}
Now we will analyze $N_{k}^{2}-N_{k+1}^{2}$. We have
\begin{equation}
\label{M2b}
N_{k}^{2}-N_{k+1}^{2}=(N_{k+1}-N_{k})^{2}+2N_{k+1}(N_{k}-N_{k+1}).
\end{equation}
Let us prove that $\E[N_{k+1}(N_{k}-N_{k+1})\arrowvert N_{k}]\geq 0$ using the set $F_{k}$ defined by \eqref{Fk}. It holds that
$$
\begin{array}{l}
\displaystyle
\E[N_{k+1}(N_{k}-N_{k+1})\arrowvert N_{k}]
\\[10pt]
\qquad \displaystyle =\half\E[N_{k+1}(N_{k}-N_{k+1})\arrowvert N_{k}\wedge F_{k}]+\half\E[N_{k+1}(N_{k}-N_{k+1})\arrowvert N_{k}\wedge F_{k}^{c}]
\\[10pt]
\qquad \displaystyle =
\half \Big[(|x_{k}|-\eps +\frac{\eps^{3}}{2^{k}}+\frac{\eps^{3}}{2^{k+1}})(|x_{k}|+\frac{\eps^{3}}{2^{k}}-|x_{k}|+\eps-\frac{\eps^{3}}{2^{k}}-\frac{\eps^{3}}{2^{k+1}})\Big] \\[10pt]
\qquad \displaystyle \qquad +
\half \Big[(|x_{k+1}|+\frac{\eps^{3}}{2^{k+1}})(|x_{k}|+\frac{\eps^{3}}{2^{k}}-|x_{k+1}|-\frac{\eps^{3}}{2^{k+1}})\Big]
\\[10pt]
\qquad \displaystyle 
\geq
\half \Big(|x_{k}|-\eps+\frac{\eps^{3}}{2^{k}}+\frac{\eps^{3}}{2^{k+1}}\Big)\Big(\eps-\frac{\eps^{3}}{2^{k+1}}\Big)
\\[10pt]
\qquad \displaystyle  \qquad +\half \Big[(|x_{k}|-\eps+\frac{\eps^{3}}{2^{k+1}})(|x_{k}|+\frac{\eps^{3}}{2^{k}}-
|x_{k}|-\eps-\frac{\eps^{3}}{2^{k+1}})\Big]
\end{array} 
$$
here we used that  $|x_{k}|-\eps\leq|x_{k+1}|\leq|x_{k}|+\eps$. Thus
$$\E[N_{k+1}(N_{k}-N_{k+1})\arrowvert N_{k}]\geq\half(|x_{k}|-\eps+\frac{\eps^{3}}{2^{k+1}}+\frac{\eps^{3}}{2^{k}})(\eps-\frac{\eps^{3}}{2^{k+1}})+\half(|x_{k}|-\eps+\frac{\eps^{3}}{2^{k+1}})(-\eps+\frac{\eps^{3}}{2^{k+1}}),$$
and then
$$\E[N_{k+1}(N_{k}-N_{k+1})\arrowvert N_{k}]\geq\half \Big[\frac{\eps^{3}}{2^{k}}(\eps-\frac{\eps^{3}}{2^{k+1}})\Big]\geq 0.$$
If we go back to \eqref{M2b} and use \eqref{e/3b} and the result we have just obtained we arrive to
$$\E[N_{k}^{2}-N_{k+1}^{2}\arrowvert N_{k}]\geq\E[(N_{k+1}-N_{k})^{2}\arrowvert N_{k}]\geq \frac{\eps^{2}}{3}.$$
Therefore, for the sequence of random variables
$$\W_{k}=N^{2}_{k}+\frac{k\eps^{2}}{3}$$
we have
$$\E[\W_{k}-\W_{k+1}\arrowvert \W_{k}]=\E[N_{k}^{2}-N_{k+1}^{2}-\frac{\eps^{2}}{3}\arrowvert \W_{k}]\geq 0.$$
As  $\E[\W_{k}\arrowvert\W_{k}]=\W_{k}$ then
$$\E[\W_{k+1}\arrowvert \W_{k}]\leq\W_{k},$$
that is, the sequence $\{\W_{k}\}_{k\geq 1}$ is a \textit{supermartingale}. In order to use the \textit{OSTh},
given a fixed integer $m\in\N$ we define the stopping time
$$ \tau_{m}=\tau \wedge m := \min\{\tau,m\}$$
Now this new stopping time verifies
$ \tau_{m}\leq m $
which is the first hypothesis of the \textit{OSTh}. Then, using the \textit{OSTh} we obtain
$$ \E[\W_{\tau_{m}}]\leq \W_{0}. $$
Observe that $\lim\limits_{m\rightarrow\infty}\tau\wedge m = \tau $ almost surely. Then, using \textit{Fatou's Lemma}, we arrive to
$$ \E[\W_{\tau}]=\E[\liminf_{m} \W_{\tau\wedge m}]\underbrace{\leq}_{Fatou} \liminf_{m} \E[\W_{\tau\wedge m}]\underbrace{\leq}_{OSTh} \W_{0}$$
Thus, we obtain 
$ \E[\W_{\tau}]\leq \W_{0}$,
i.e.,
\begin{equation}
\label{OST2}
\E[N^{2}_{\tau}+\frac{\tau\eps^{2}}{3}]\leq N_{0}^{2}.
\end{equation}
Then,
$$\E[\tau]\leq 3(|x_{0}|+\eps^{3})^{2}\eps^{-2}\leq 4|x_{0} |^{2}\eps^{-2}$$
if $\eps$ is small enough. 
On the other hand, if we go back to \eqref{OST2} we have
$$\E[N_{\tau}^{2}]\leq N_{0}^{2},$$
i.e.
$$\E[|x_{\tau}|^{2}]\leq\E[(|x_{\tau}|+\frac{\eps^{3}}{2^{\tau}})^{2}]\leq (|x_{0}|+\eps^{3})^{2}\leq 2 |x_{0}|^{2}.$$
What we have so far is that
\begin{equation}
\label{etau}
\E[\tau]\leq 4|x_{0}|^{2}\eps^{-2}
\qquad \mbox{and} \qquad \E[|x_{\tau}|^{2}]\leq 2 |x_{0}|^{2}.
\end{equation}
We will use these two estimates to prove 
$$ \P \Big(\tau \geq \frac{a}{\eps^2}\Big)< \eta 
\qquad \mbox{and} \qquad \P\Big(|x_{\tau} |\geq a\Big)< \eta. $$
Given $\eta > 0$ and $a > 0$, we take $x_{0}\in\Om$ such that $|x_{0} |< r_{0}$ with $r_{0}$ that will be choosed later
(depending on $\eta$ and $a$). We have
$$ C r_{0}^{2}\eps^{-2}\geq C |x_{0}-y |^{2}\eps^{-2} \geq \E^{x_{0}}[\tau]\geq \P \Big(\tau \geq \frac{a}{\eps^{2}}\Big)\frac{a}{\eps^{2}}.$$
Thus
$$ \P(\tau \geq \frac{a}{\eps^{2}})\leq C \frac{r_{0}^{2}}{a}< \eta$$
which holds true if $r_{0}<\sqrt{\frac{\eta a}{C}}$.

Also we have
$$ C r_{0}^{2} \geq C |x_{0} |^{2}\geq \E^{x_{0}}[|x_{\tau} |^{2}]\geq a^{2}\P(|x_{\tau} |^{2}\geq a^{2}).$$
Then 
$$ \P(|x_{\tau} |\geq a)\leq C\frac{r_{0}^{2}}{a^{2}}< \eta$$
which holds true if $r_{0}< \sqrt{\frac{\eta a^{2}}{C}}$. Observe that if we take $a<1$ we have $\sqrt{\frac{\eta a^{2}}{C}}<\sqrt{\frac{\eta a}{C}}$, then if we choose $r_{0}<\sqrt{\frac{\eta a^{2}}{C}}$ both conditions are fulfilled at the same time.
\end{proof}

\subsection{Estimates for the Random Walk game}

In this case we are going to assume that we are permanently playing on board 2, with the random walk game.
The estimates for this game follow the same ideas as before, and are even simpler since there are no strategies 
of the players involved in this case. We include the details for completeness. 
\begin{lemma}
Given $\eta>0$ and $a>0$, there exists $r_{0}>0$ and $\eps_{0}>0$ such that, given $y\in\partial\Om$ and $x_{0}\in\Om$ with $|x_{0}-y|<r_{0}$, if we play random we obtain 
$$\mathbb{P} \Big(|x_{\tau}-y|< a \Big) \geq 1 - \eta
\qquad \mbox{and} \qquad\mathbb{P} \Big(\tau \geq \frac{a}{\eps^2} \Big)< \eta$$
  for $\eps<\eps_{0}$ and $x_{\tau}\in\R^{N}\backslash\Om$ the first position outside $\Om$.
\end{lemma}
\begin{proof}
Recall that we assumed that $ \Om $ satisfies the uniform exterior ball property for a certain $ \theta_ {0}> 0 $. 

For $ N \geq 3 $, given $ \theta <\theta_{0} $, and $ y \in \Om $ we are going to assume that $ z_{y} = 0 $ is chased so that we have 
$ \ol {B_{\theta}(0)}  \cap \ol {\Om} = \{ y \}$. We define the set
 $$\Om_{\eps}=\{x\in\R^{N}:d(x,\Om)<\eps\}$$
 for $ \eps $ small enough. Now, we consider the function $\mu:\Om_{\eps}\rightarrow\R$ given by 
 \begin{equation}
 \label{mu}
 \mu(x)=\frac{1}{\theta^{N-2}}-\frac{1}{|x |^{N-2}}.
 \end{equation}
This function is positive in $ \ol{\Om}\backslash \{ y \} $,
radially increasing and harmonic in $ \Om $. Also it holds that $ \mu (y) = 0 $.
For $N=2$ we take 
$ \mu(x)=\ln (\theta )-\ln (|x |)$
 and we leave the details to the reader.

We will take the first position of the game, $ x_{0} \in \Om $, such that $ |x_{0} -y |<r_{0} $ with $ r_{0} $ 
to be choosed later. Let $ (x_{k})_{k \geq 0} $ be the sequence of positions of the game playing random walks. 
Consider the sequence of random variables
$$ N_{k}=\mu(x_{k})$$
for $ k \geq 0 $. Let us prove that $ N_{k} $ is a \textit{martingale}. Indeed
 $$ \E [N_{k + 1} \arrowvert N_{k}] = \kint_{B _{\eps} (x_{k})} \mu(y) dy = \mu (x_{k}) = N_{k}. $$
 Here we have used that $ \mu $ is harmonic. Since $ \mu $ is bounded in $\Omega$, 
 the third hypothesis of {\it OSTh} is fulfilled, hence we obtain
 \begin{equation}
 \label{muxo}
 \E [\mu(x_{\tau})] = \mu(x_ {0}).
 \end{equation}
Let us estimate the value $\mu(x_{0})$
\begin{equation}
\label{acotmu}
\mu(x_{0})=\frac{1}{\theta^{n-2}}-\frac{1}{|x_{0} |^{n-2}}=\frac{|x_{0}|^{n-2}-\theta^{n-2}}{\theta^{n-2}|x_{0}|^{n-2}}=\frac{(|x_{0}|-\theta)}{\theta^{n-2}|x_{0}|^{n-2}}\Big(\sum\limits_{j=1}^{N-2}|x_{0}|^{N-2-j}\theta^{j-1}\Big).
\end{equation}
The first term can be bounded as
$$(|x_{0}|-\theta)=(|x_{0}|-|y |)\leq |x_{0}-y|<r_{0}.$$
To deal with the second term we will ask $\theta<1$ and $|x_{0}|^{l}\leq R^{N-2}$ where $R=\max_{x\in\Om}\{|x |\}$
(suppose $R>1$). Then, we obtain 
$$\sum\limits_{j=1}^{N-2}|x_{0}|^{n-2-j}\theta^{j-1}\leq R^{N-2}(N-2).$$
Finally, we will use that $|x_{0}|>\theta$. 
Plugging all these estimates in \eqref{acotmu} we obtain
$$\mu(x_{0})\leq r_{0}(\frac{R^{N-2}(N-2)}{\theta^{2(N-2)}}).$$
If we call $c(\Om,\theta)=\frac{R^{N-2}(N-2)}{\theta^{2(N-2)}}$ and come back to \eqref{muxo} we get
\begin{equation}
\label{espmu}
\E[\mu(x_{\tau})]<c(\Om,\theta)r_{0}.
\end{equation}
We need to establish a relation between $ \mu (x_{\tau}) $ and $ |x_{\tau} -y |$. To this end, we take the function $ b: [\theta, + \infty) \rightarrow \R $ given by
\begin{equation}
\label{funb}
b(\ol{a})=\frac{1}{\theta^{N-2}}-\frac{1}{\ol{a}^{N-2}}.
\end{equation}
Note that this function is the radial version of $ \mu $. It is positive and increasing,
then, it has an inverse (also increasing) that is given by the formula
$$ \ol{a}(b)=\frac{\theta}{(1-\theta^{N-2}b)^{\frac{1}{N-2}}}.$$
This function is positive, increasing and convex, since $ \ol{a}''> 0 $. Then for $ b <1 $ we obtain 
\begin{equation}
\label{bmay}
\ol{a}(b)\leq \theta + (\ol{a}(1)-\theta)b.
\end{equation}
Let us call $K(\theta)=(\ol{a}(1)-\theta)>0$ (this constant depends only on $\theta$). Using the relationship between $\ol{a}$ and $b$ we 
obtained the following: given $ \ol{a}> \theta $ there is $ b> 0 $ such that
$$ 
\mbox{if } \mu(x_{\tau})<b \mbox{ then } |x_{\tau} |<\ol{a}.
$$
Here we are using that the function $ b (\ol {a}) $ is increasing. Now one can check that,
for all $ a> 0 $ , there are $ \ol{a}> \theta $ and $ \eps_{0}> 0 $ such that, if
$$|x_{\tau}|< \ol{a} \qquad \mbox{and} \qquad d(x_{\tau},\Om)<\eps_{0},$$
then
$$ |x_{\tau}-y|<a.$$

Putting everything together we obtained that,
given $a>0$, exist $\ol{a}>\theta$, $b>0$ and $\eps_{0}>0$ such that  
$$
\mbox{if } \mu(x_{\tau})<b \Rightarrow |x_{\tau}-y|<a \ , \ d(x_{\tau},\Om)<\eps_{0}.
$$
We ask for
$
0<b<a
$
that we will used later. Then, we have
$$\P(\mu(x_{\tau})\geq b)\geq \P(|x_{\tau}-y|\geq a).$$
Coming back to \eqref{espmu} we get  
\begin{equation}
\label{desb}
c(\Om,\theta)r_{0}>\E[\mu(x_{\tau})]\geq \P(\mu(x_{\tau})\geq b)b\geq \P(|x_{\tau}-y|\geq a)b
\end{equation}
Using that $\ol{a}-\theta \leq K(\theta)b$ we obtain 
$$c(\Om,\theta)r_{0}>\P(|x_{\tau}-y|\geq a)\frac{\ol{a}-\theta}{K(\theta)}$$
Then
\begin{equation}
\label{desnorma}
\P(|x_{\tau}-y|\geq a)<\frac{c(\Om,\theta)r_{0}K(\theta)}{\ol{a}-\theta}<\eta
\end{equation}
which holds true if
$$r_{0}<\frac{\eta(\ol{a}-\theta)}{c(\Om,\theta)K(\theta)}.$$ 
This is one of the inequalities we wanted to prove.

Now let us compute
\begin{equation}
\label{ENK}
\E[N_{k+1}^{2}-N_{k}^{2}\arrowvert N_{k}]=\kint_{B_{\eps}(x_{k})}(\mu^{2}(w)-\mu^{2}(x_{k}))dw.
\end{equation}
Let us call $\varphi=\mu^{2}$. If we made the Taylor expansion of order two we obtain
$$ \varphi(w)=\varphi(x_{k})+\langle\nabla\varphi(x_{k}),(w-x_{k})\rangle+\half\langle D^{2}\varphi(x_{k})(w-x_{k}),(w-x_{k})\rangle+O(|w-x_{k}|^{3}).$$
Then
$$
\begin{array}{l}
\displaystyle 
\kint_{B_{\eps}(x_{k})}(\varphi(w)-\varphi(x_{k}))dw
 \displaystyle =\kint_{B_{\eps}(x_{k})}\langle\nabla\varphi(x_{k}),(w-x_{k})\rangle dw
\\[10pt]
\qquad \qquad \displaystyle +\half\kint_{B_{\eps}(x_{k})}\langle D^{2}\varphi(x_{k})(w-x_{k}),(w-x_{k})\rangle dw
 \displaystyle +\kint_{B_{\eps}(x_{k})}O(|w-x_{k}|^{3})dw.
\end{array} 
$$
Let us analyze these integrals 
$$\kint_{B_{\eps}(x_{k})}\langle\nabla\varphi(x_{k}),(w-x_{k})\rangle dw=0.$$
On the other hand, for $\langle D^{2}\varphi(x_{k})(w-x_{k}),(w-x_{k})\rangle$,
changing variables as
$w=x_{k}+\eps z$,
 it holds that 
$$\kint_{B_{\eps}(x_{k})}\langle D^{2}\varphi(x_{k})(w-x_{k}),(w-x_{k})\rangle dw
=\sum\limits_{j=1}^{N}\partial_{x_{j}x_{j}}^{2}\varphi(x_{k})\eps^{2}\kint_{B_{1}(0)}z_{j}^{2}dz=\kappa\eps^{2}\sum\limits_{j=1}^{N}\partial_{x_{j}x_{j}}^{2}\varphi(x_{k}).$$
Here we find the constant $\kappa $ that appears in the second equation in \eqref {ED1}. Let us
compute the second derivatives of $\varphi$. As $\varphi=\mu^{2}$,
 $$\sum\limits_{j=1}^{N}\partial_{x_{j}x_{j}}^{2}\varphi(x_{k})=2\sum\limits_{j=1}^{N}(\partial_{x_{j}}\mu(w))^{2}+2\mu(x_{k})\sum\limits_{j=1}^{n}\partial_{x_{j}x_{j}}^{2}\mu(x_{k}).$$
The second term is zero because $\mu$ is harmonic in $\Om$. Hence, we arrived to  
$$\sum\limits_{j=1}^{N}\partial_{x_{j}x_{j}}^{2}\varphi(x_{k})=2\sum\limits_{j=1}^{N}(\partial_{x_{j}}\mu(w))^{2}.$$
Using the definition of $\mu$ \eqref{mu} we get
$$\sum\limits_{j=1}^{N}\partial_{x_{j}x_{j}}^{2}\varphi(x_{k})=\frac{2(N-2)^{2}}{|x_{k} |^{2(N-2)}}.$$
Putting everything together 
$$
\begin{array}{l}
\displaystyle 
\kint_{B_{\eps}(x_{k})}(\varphi(w)-\varphi(x_{k}))dw
=\half\kappa\eps^{2}\frac{2(N-2)^{2}}{|x_{k} |^{2(N-2)}}+O(|w-x_{k}|^{3})
 \geq \eps^{2}\frac{\kappa(N-2)^{2}}{R^{2(n-2)}}-\gamma\eps^{3}\displaystyle \geq \eps^{2}\frac{\kappa(N-2)^{2}}{2R^{2(N-2)}},
\end{array}
$$
if $\eps$ is small enough (here $R=\max_{x\in\Om} \{ |x |\}$). Let us call 
$$\sigma(\Om)=\frac{\kappa(N-2)^{2}}{2R^{2(N-2)}}.$$ Then, if we go back to \eqref{ENK} we get
$$ \E[N_{k+1}^{2}-N_{k}^{2}\arrowvert N_{k}]\geq \sigma(\Om)\eps^{2}.$$
Let us consider the sequence of random variables $(\W_{k})_{k\geq 0}$ given by 
$$ \W_{k}=-N_{k}^{2}+\sigma(\Om)k\eps^{2}.$$
Then
$$ \E[\W_{k+1}-\W_{k}\arrowvert \W_{k}]=\E[-(N_{k+1}^{2}-N_{k}^{2})+\sigma\eps^{2}\arrowvert N_{k}]\leq 0$$
That is, $\W_{k}$ is a \textit{supermartingale}. Using the \textit{OSTh} in the same way as before we get
$$\E[-\mu^{2}(x_{\tau})+\sigma\tau\eps^{2}]\leq -\mu^{2}(x_{0}).$$
Therefore,
\begin{equation}
\label{paratau}
\E[\sigma\tau\eps^{2}]\leq -\mu^{2}(x_{0})+\E[\mu^{2}(x_{\tau})]\leq \E[\mu^{2}(x_{\tau})].
\end{equation}
Hence, we need a bound for $\E[\mu^{2}(x_{\tau})]$. We have
$$\E[\mu^{2}(x_{\tau})]=\E[\mu^{2}(x_{\tau})\arrowvert\mu(x_{\tau})<b]\P(\mu(x_{\tau})<b)+\E[\mu^{2}(x_{\tau})\arrowvert\mu(x_{\tau})\geq b]\P(\mu(x_{\tau})\geq b).$$
It holds that
$\E[\mu^{2}(x_{\tau})\arrowvert\mu(x_{\tau})<b]\leq b^{2}$ and $\P(\mu(x_{\tau})<b)\leq 1$. 
If we call $M(\eps_{0})=\max_{x\in\Om_{\eps_{0}}}\arrowvert\mu(x)\arrowvert$ it holds  $\E[\mu^{2}(x_{\tau})\arrowvert\mu(x_{\tau})\geq b]\leq M(\eps_{0})^{2}$. Finaly using \eqref{desb} we obtain $\P(\mu(x_{\tau})\geq b)\leq \frac{c(\Om,\theta)r_{0}}{b}$. Thus
$$
 \E[\mu^{2}(x_{\tau})]\leq b^{2}+M(\eps_{0})^{2}\frac{c(\Om,\theta)r_{0}}{b}.
 $$
 Recall that we imposed $0<b<a$. Then
\begin{equation}
\label{bmasb}
 \E[\mu^{2}(x_{\tau})]\leq a^{2}+M(\eps_{0})^{2}\frac{c(\Om,\theta)r_{0}}{b}.
\end{equation}
On the other hand, we have
$$\sigma\E[\tau\eps^{2}]\geq\P(\tau\eps^{2}\geq a)a\sigma.$$
Using \eqref{paratau} and \eqref{bmasb} we get 
$$
\P \Big(\tau\geq\frac{a}{\eps^{2}}\Big)\leq \frac{a}{\sigma}+M(\eps_{0})^{2}\frac{c(\Om,\theta)r_{0}}{b\sigma a}.
$$
If we ask $$ \frac{a}{\sigma} < \frac{\eta}{2}$$ we arrive to 
$$
\P \Big(\tau\geq\frac{a}{\eps^{2}}\Big) \leq \frac{\eta}{2}+M(\eps_{0})^{2}\frac{c(\Om,\theta)r_{0}}{b a\sigma}<\eta
$$
which is true if we impose that
$$ r_{0}< \frac{b\eta a \sigma}{2M(\eps_{0})c(\Om,\theta)}.$$
Thus we achieve the second inequality of the lemma, and the proof is finished.
\end{proof}

Now we are ready to prove the second condition in the Arzela-Ascoli type lemma.

\begin{lemma}\label{lem.ascoli.arzela.asymp} Given $\delta>0$ there are constants
    $r_0$ and $\eps_0$ such that for every $\eps < \eps_0$
    and any $x, y \in \overline{\Omega}$ with $|x - y | < r_0 $
    it holds
$$
|u^\eps (x) - u^\eps (y)| < \delta \qquad \mbox{and} \qquad |v^\eps (x) - v^\eps (y)| < \delta.
$$
\end{lemma}

\begin{proof} We deal with the estimate for $u^\eps$. 
Recall that $u^{\eps}$ is the value of the game playing in the first board (where we play Tug-of-War).
The computations for $v^\eps$ are similar.

First, we start with two close points $x$ and $y$ with $y\not\in \Omega$ and $x\in \Omega$.  
We have that $u^{\eps}(y)=\ol{f}(y)$ for $y\in\partial\Om$.
Given $\eta >0$ we take $a$, $r_{0}$, $\eps_{0}$ and $S^{*}_{I}$ the strategy as in Lemma \ref{lema.estim.ToW}.
Let 
$$ A=\Big\{\mbox{the position does not change board in the first } \ \lceil \frac{a}{\eps^{2}}\rceil \mbox{ plays and } \tau < \lceil \frac{a}{\eps^{2}}\rceil \Big\}.$$

We consider two cases.

\textbf{1st case:} We are going to show that $u^{\eps}(x_{0})-\ol{f}(y) \geq - A(a,\eta)$ with $A(a,\eta)\searrow 0$ if $a\rightarrow 0$ and 
$\eta\rightarrow 0$.
We have
$$u^{\eps}(x_{0})\geq \inf_{S_{II}}\E^{x_{0}}_{S^{*}_{I},S_{II}}[h(x_{\tau})].$$
Now 
$$
\begin{array}{l}
\displaystyle 
\E^{x_{0}}_{S^{*}_{I},S_{II}}[ h (x_{\tau})]  = \E^{x_{0}}_{S^{*}_{I},S_{II}}[h( x_{\tau})\arrowvert A]\P(A)
 +  \E^{x_{0}}_{S^{*}_{I},S_{II}}[h( x_{\tau})\arrowvert A^{c}]\P(A^{c})
\\[10pt]
\qquad \displaystyle \geq \E^{x_{0}}_{S^{*}_{I},S_{II}}[\ol{f}(x_{\tau})\arrowvert A]\P(A)-\max\{\lvert\ol{f}|,\lvert\ol{g}|\}\P(A^{c}).
\end{array}
$$
Now we estimate $\P(A)$ and $\P(A^{c})$. We have that  
$$\P(A^{c})\leq \P \Big(\mbox{the game changes board before }\lceil \frac{a}{\eps^{2}}\rceil \mbox{ plays}\Big)
+\P(\tau\geq\lceil \frac{a}{\eps^{2}}\rceil).$$
Hence we are left with two bounds. First, we have 
 \begin{equation}
 \label{Ac1}
  \P \Big(\mbox{the game changes board before }\lceil \frac{a}{\eps^{2}}\rceil \mbox{ plays} \Big)=1-(1-\eps^{2})^{\frac{a}{\eps^{2}}} \leq (1-e^{-a})+\eta
 \end{equation}
 for $\eps$ small enough.
 Here we are using that $(1-\eps^{2})^{\frac{a}{\eps^{2}}}\nearrow e^{-a}$.
 
Now, we observe that using Lemma \ref{lema.estim.ToW} we get
\begin{equation}
\label{Ac2}
\P\Big(\tau \geq \frac{a}{\eps^{2}}\Big) \leq \P\Big(\tau \geq \frac{a}{\eps_{0}^{2}}\Big)\leq \eta,
\end{equation}
for $\eps < \eps_{0}$.
 From \eqref{Ac1} and \eqref{Ac2} we obtain
$$ \P(A^{c})\leq (1-e^{-a})+\eta +\eta= (1-e^{-a})+2\eta $$
and hence
 $$ \P(A) =1-\P(A^{c}) \geq 1-[(1-e^{-a})+2\eta] .$$
 Then we obtain
 \begin{equation}
 \label{arriba}
 \begin{array}{l}
\displaystyle  \E^{x_{0}}_{S^{*}_{I},S_{II}}[h(x_{\tau})] 
\displaystyle \geq \E^{x_{0}}_{S^{*}_{I},S_{II}}[\ol{f}(x_{\tau})\arrowvert A] (1-[(1-e^{-a})+2\eta])-\max\{\lvert\ol{f}|,\lvert\ol{g}|\}[(1-e^{-a})+2\eta] .
\end{array}
 \end{equation}
Let us analyze the expected value $\E^{x_{0}}_{S^{*}_{I},S_{II}}[\ol{f}(x_{\tau})\arrowvert A]$. 
Again we need to consider two events,
$$ A_{1}=A\cap \{ |x_{\tau}-y|< a \} \qquad \mbox{and} \qquad A_{2}=A\cap \{ |x_{\tau}-y|\geq a\}.$$
We have that $ A=A_{1}\cup A_{2}$.
Then 
\begin{equation}
\label{retomo}
  \E^{x_{0}}_{S^{*}_{I},S_{II}}[\ol{f}(x_{\tau})\arrowvert A]=\E^{x_{0}}_{S^{*}_{I},S_{II}}[\ol{f}(x_{\tau})\arrowvert A_{1}]\P(A_{1})+\E^{x_{0}}_{S^{*}_{I},S_{II}}[\ol{f}(x_{\tau})\arrowvert A_{2}]\P(A_{2}).
\end{equation}
Now we observe that
\begin{equation}
\label{a2}
\P(A_{2})\leq \P( |x_{\tau}-y|\geq a)\leq \eta .
\end{equation}
To get a bound for the other case we observe that
$ A_{1}^{c}=A^{c}\cup  \{ |x_{\tau}-y|\geq a\}$.
Therefore 
$$ \P(A_{1})=1-\P(A_{1}^{c})\geq 1-[\P(A^{c})+\P(|x_{\tau}-y|\geq a)],$$
and we arrive to
 \begin{equation}
 \label{a1}
 \P(A_{1})\geq 1-[(1-e^{-a})+2\eta+\eta]=1-[(1-e^{-a})+3\eta].
 \end{equation}
If we go back to \eqref{retomo} and use \eqref{a1} and \eqref{a2} we get 
\begin{equation}
\label{retomo2}
 \E^{x_{0}}_{S^{*}_{I},S_{II}}[\ol{f}(x_{\tau})\arrowvert A]\geq\E^{x_{0}}_{S^{*}_{I},S_{II}}[\ol{f}(x_{\tau})\arrowvert A_{1}](1-[(1-e^{-a})+3\eta])-\max \{\lvert\ol{f}|\}\eta .
\end{equation}
Using that $\ol{f}$ is Lipschitz we obtain 
$$ \ol{f}(x_{\tau})\geq \ol{f}(y)-L|x_{\tau}-y|\geq \ol{f}(y)-La ,$$
and then we obtain (using that $(\ol{f}(y)-La)$ does not depend on the strategies)  
  \begin{equation}
\E^{x_{0}}_{S^{*}_{I},S_{II}}[\ol{f}(x_{\tau})\arrowvert A]\geq(\ol{f}(y)-La)(1-[(1-e^{-a})+3\eta])-\max \{\lvert\ol{f}|\}\eta.
   \end{equation}
Recalling \eqref{arriba} we obtain
 $$
 \begin{array}{l}
 \displaystyle  \E^{x_{0}}_{S^{*}_{I},S_{II}}[h ( x_{\tau})] \displaystyle \\[10pt]
\quad  \geq ((\ol{f}(y)-La)(1-[(1-e^{-a})+3\eta]) -\max \{\lvert\ol{f}|\}\eta ) (1-[(1-e^{-a})+2\eta]) 
 \displaystyle  -\max\{\lvert\ol{f}|,\lvert\ol{g}|\}[(1-e^{-a})+2\eta].
 \end{array}
 $$ 
Notice that when $\eta \rightarrow 0$ and $a\rightarrow 0$ the the right hand side
goes to $\ol{f}(y)$, hence we have obtained
 $$ \E^{x_{0}}_{S^{*}_{I},S_{II}}[h( x_{\tau})]\geq \ol{f}(y)- A(a,\eta)$$
 with $A(a,\eta)\to 0$. Taking the infimum over all possible strategies $S_{II}$ we get 
 $$ u^{\eps}(x_{0})\geq \ol{f}(y)- A(a,\eta)$$ 
 with $A(a,\eta)\to 0$ as $\eta\rightarrow 0$ and $a\rightarrow 0$ as we wanted to show.

\textbf{2nd case:} Now we want to show that $u^{\eps}(x_{0})-\ol{f}(y)\leq B(a,\eta)$ with $B(a,\eta)\searrow 0$ as $\eta\rightarrow 0$ and 
$a\rightarrow 0$.
In this case we just use the strategy $S^*$ from Lemma \ref{lema.estim.ToW} as the strategy for the second player
$S^{*}_{II}$ and we obtain
$$ u^{\eps}(x_{0})\leq \sup_{S_{II}}\E^{x_{0}}_{S_{I},S^{*}_{II}}[h(x_{\tau})].$$
Using again the set $A$ that we considered in the previous case we obtain
$$ \E^{x_{0}}_{S_{I},S^{*}_{II}}[ h( x_{\tau})]= \E^{x_{0}}_{S_{I},S^{*}_{II}}[\ol{f}( x_{\tau})\arrowvert A]\P(A)+ 
\E^{x_{0}}_{S_{I},S^{*}_{II}}[h( x_{\tau})\arrowvert A^{c}]\P(A^{c}).$$
We have that $\P(A) \leq 1$ and $\P(A^{c})\leq (1-e^{-a})+2\eta $. Hence we get
\begin{equation}
\label{retomo3}
\E^{x_{0}}_{S_{I},S^{*}_{II}}[h( x_{\tau})]\leq \E^{x_{0}}_{S_{I},S^{*}_{II}}[\ol{f}( x_{\tau})\arrowvert A]+\max\{\lvert\ol{f}|,\lvert\ol{g}|\}[(1-e^{-a})+2\eta].
\end{equation}
To bound $\E^{x_{0}}_{S_{I},S^{*}_{II}}[\ol{f}(x_{\tau})\arrowvert A]$ we will use again the sets $A_{1}$ and $A_{2}$ 
as in the previous case. We have
$$ \E^{x_{0}}_{S_{I},S^{*}_{II}}[\ol{f}(x_{\tau})\arrowvert A]=\E^{x_{0}}_{S_{I},S^{*}_{II}}[\ol{f}(x_{\tau})\arrowvert A_{1}]\P(A_{1})+\E^{x_{0}}_{S_{I},S^{*}_{II}}[\ol{f}(x_{\tau})\arrowvert A_{2}]\P(A_{2}).$$
Now we use that $\P(A_{1}) \leq 1$ and $\P(A_{2})\leq c\eta$ to obtain
$$ \E^{x_{0}}_{S_{I},S^{*}_{II}}[\ol{f}(x_{\tau})\arrowvert A]\leq \E^{x_{0}}_{S_{I},S^{*}_{II}}[\ol{f}(x_{\tau})\arrowvert A_{1}]+\max\{\lvert\ol{f}|\}\eta .$$
Now for $ \E^{x_{0}}_{S_{I},S^{*}_{II}}[\ol{f}(x_{\tau})\arrowvert A_{1}]$ we use that $\ol{f}$ is Lipschitz to obtain
$$ \E^{x_{0}}_{S_{I},S^{*}_{II}}[\ol{f}(x_{\tau})\arrowvert A]\leq \E^{x_{0}}_{S_{I},S^{*}_{II}}[\ol{f}(y)+La\arrowvert A_{1}]+\max\{\lvert\ol{f}|\}\eta .$$
As $(\ol{f}(y)+La)$ does not depend on the strategies we have
$$ \E^{x_{0}}_{S_{I},S^{**}_{II}}[\ol{f}(x_{\tau})\arrowvert A]\leq(\ol{f}(y)+La)+\max\{\lvert\ol{f}|\}\eta ,$$
and therefore we conclude that 
$$ \E^{x_{0}}_{S_{I},S^{*}_{II}}[h( x_{\tau} )]\leq \ol{f}(y)+La+\max\{\lvert\ol{f}|\}\eta + \max\{\lvert\ol{f}|,\lvert\ol{g}|\}[(1-e^{-a})+2\eta].
$$
We have proved that 
$$ \E^{x_{0}}_{S_{I},S^{*}_{II}}[ h( x_{\tau})]\leq \ol{f}(y) + B(a,\eta) $$
with $B(a,\eta)\to 0$. Taking supremum over the strategies for Player I we obtain 
$$ u^{\eps}(x_{0})\leq \ol{f}(y)+B(a,\eta)$$
with $B(a,\eta)\rightarrow 0$ as $\eta\rightarrow 0$ and $a\rightarrow 0$. 

Therefore, we conclude that 
$$ |u^{\eps}(x_{0})-\ol{f}(y)|< \max\{A(a,\eta),B(a,\eta)\},$$
that holds when $y \not \in\Omega$ and $x_0$ is close to $y$.

An analogous estimate holds for $v^\eps$. 

Now, given two points $x_0$ and $z_0$ inside $\Omega$ with $|x_0-z_0|<r_0$ we couple the  
game starting at $x_0$ with the game starting at $z_0$ making the same movements and also
changing board simultaneously. This coupling generates two sequences of positions $(x_i,j_i)$ and $(z_i,k_i)$
such that $|x_i - z_i|<r_0$ and $j_i=k_i$ (since they change boars at the same time both games are at
the same board at every turn). This continues until one of the games exits the domain (say at $x_\tau \not\in \Omega$).
At this point for the game starting at $z_0$ we have that its position $z_\tau$ is close to the exterior point $x_\tau \not\in \Omega$ (since we
have $|x_\tau - z_\tau|<r_0$) and hence we can use our previous estimates for points close to the boundary to conclude that 
$$ |u^{\eps}(x_{0})- u^\eps (z_0)|< \delta, \qquad \mbox{ and }
\qquad |v^{\eps}(x_{0})- v^\eps (z_0)|< \delta. $$
This ends the proof. 
 \end{proof}
 
 As a consequence, we have convergence of $(u^{\eps},v^{\eps})$ as $\eps \to 0$ along subsequences. 
 
 \begin{theorem} \label{teo.conv.unif}
 Let $(u^{\eps},v^{\eps})$ be solutions to the DPP, then there exists a subsequence
 $\eps_k \to 0$ and a pair on functions $(u,v)$ continuous in $\overline{\Omega}$ such that  
 $$
 u^{\eps_k} \to u, \qquad \mbox{ and } \qquad  v^{\eps_k} \to v, 
 $$ 
 uniformly in  $\overline{\Omega}$.
 \end{theorem}
 
 \begin{proof} Lemma \ref{lem.ascoli.arzela.acot} and
 Lemma \ref{lem.ascoli.arzela.asymp} imply that we can use 
 the Arzela-Ascoli type lemma, Lemma \ref{lem.ascoli.arzela}. 
 \end{proof}

\section{Existence of viscosity solutions} \label{sect-limiteviscoso}

Now, we prove that any possible uniform limit of $(u^\eps,v^\eps)$ is a viscosity solution to
the limit PDE problem \eqref{ED1}. 

\begin{theorem} \label{teo.converge.222}
Any uniform limit of the values of the game $(u^\eps,v^\eps)$, $(u,v)$, is a viscosity
solution to
\begin{equation}
\label{ED1.th}
 \left\lbrace
\begin{array}{ll}
- \displaystyle \half \Delta_{\infty}u(x) + u(x) - v(x)=0 \qquad &  \ x \in \Omega,  \\[10pt]
-   \displaystyle \frac{\kappa}{2} \Delta v(x) + v(x) - u(x)=0  \qquad &  \ x \in \Omega,  \\[10pt]
u(x) = f(x) \qquad & \ x \in \partial \Omega,  \\[10pt]
v(x) = g(x) \qquad & \ x \in \partial \Omega.
\end{array}
\right.
\end{equation}
\end{theorem}

\begin{proof} Since $u^\eps =\ol{f}$ and $v^\eps =\ol{g}$ in $\R^N \setminus \Omega$ we have that 
$u = f$ and $v= g$ on $\partial \Omega$.

{\it Infinity Laplacian.} let us start by showing that $u$ is a viscosity subsolution to 
$$ - \half \Delta_{\infty}u(x)+u(x)-v(x)=0. $$
Let $x_{0} \in \Om$ and $\phi \in {C}^{2}(\Om)$ auch that $u(x_{0})-\phi(x_{0})=0$ and $u-\phi$ has an absolute 
maximum at $x_{0}$. Then, there exists a sequence $(x_{\eps})_{\eps >0}$ with $ x_{\eps} \rightarrow x_{0}$ as $\eps \rightarrow 0$ 
verifying  
$$
u^{\eps}(y)-\phi(y)\leq u^{\eps}(x_{\eps})-\phi(x_{\eps})+\eps^{3}.
$$
Then we obtain
\begin{equation}
\label{ast}
 u^{\eps}(y)-u^{\eps}(x_{\eps}) \leq \phi(y)-\phi(x_{\eps})+\eps^{3}
\end{equation}

Now, using the DPP, we get 
$$
u^{\eps}(x_{\eps})=\eps^{2}v^{\eps}(x_{\eps})+(1-\eps^{2})
\left\{\half \sup_{y \in B_{\eps}(x_{\eps})}u^{\eps}(y) + \half \inf_{y \in B_{\eps}(x_{\eps})}u^{\eps}\right\}
$$
and hence  
$$ 0=\eps^{2}(v^{\eps}(x_{\eps})-u^{\eps}(x_{\eps}))+(1-\eps^{2})
\left\{\half \sup_{y \in B_{\eps}(x_{\eps})}(u^{\eps}(y)-u^{\eps}(x_{\eps})) 
+ \half \inf_{y \in B_{\eps}(x_{\eps})}(u^{\eps}(y)-u^{\eps}(x_{\eps}))\right\}.$$
Using \eqref{ast} and that $\phi$ is smooth we obtain
\begin{equation}
\label{grad}
 0 \leq \eps^{2}(v^{\eps}(x_{\eps})-u^{\eps}(x_{\eps}))+(1-\eps^{2})\left\{\half \max_{y \in \ol{B}_{\eps}(x_{\eps})}(\phi(y)-\phi(x_{\eps})) + \half \min_{y \in \ol{B}_{\eps}(x_{\eps})}(\phi(y)-\phi(x_{\eps}))\right\}+\eps^{3}.
\end{equation}

Now, assume that $\nabla \phi (x_0) \neq 0$. Then, by continuity $\nabla \phi \neq 0$ in a ball $B_{r}(x_{0})$
for $r$ small. In particular, we have $\nabla \phi(x_{\eps}) \neq 0$. Call $w_{\eps}= \frac{\nabla \phi(x_{\eps})}{|\nabla \phi(x_{\eps}) |}$
and let $z_{\eps}$ with $| z_{\eps}|=1$ be such that 
$$
\max_{y\in\partial B_{\eps}(x_{\eps})}\phi(y)=\phi(x_{\eps}+\eps z_{\eps}).
$$
We have
$$
\begin{array}{l}
\displaystyle 
\phi(x_{\eps}+\eps z_{\eps})-\phi(x_{\eps}) =\eps\langle\nabla\phi(x_{\eps}),z_{\eps}\rangle+o(\eps)
 \displaystyle
\leq \eps\langle\nabla\phi(x_{\eps}),w_{\eps}\rangle+o(\eps)
=\phi(x_{\eps}+\eps w_{\eps})-\phi(x_{\eps})+o(\eps).
\end{array}
$$
On the other hand 
$$
\phi(x_{\eps}+\eps w_{\eps})-\phi(x_{\eps})=\eps\langle\nabla\phi(x_{\eps}),w_{\eps}\rangle+o(\eps)\leq \phi(x_{\eps}+\eps z_{\eps})-\phi(x_{\eps}).
$$
Therefore, we get 
$$
\eps\langle\nabla\phi(x_{\eps}),w_{\eps}\rangle+o(\eps)\leq \eps\langle\nabla\phi(x_{\eps}),z_{\eps}\rangle+o(\eps)\leq \eps\langle\nabla\phi(x_{\eps}),w_{\eps}\rangle+o(\eps).
$$
multiplying by $\eps^{-1}$ and taking the limit we arrive to
$$
\langle\nabla\phi(x_{0}),w_{0}\rangle=\langle\nabla\phi(x_{0}),z_{0}\rangle
$$
with $w_{0}=\frac{\nabla \phi(x_{0})}{|\nabla \phi(x_{0}) |}$
and we conclude that  
$$
z_{0}=w_{0}=\frac{\nabla \phi(x_{0})}{|\nabla \phi(x_{0}) |}.
$$

Going back to \eqref{grad} we obtain
\begin{equation}
\label{zeps}
0 \leq \eps^{2}(v^{\eps}(x_{\eps})-u^{\eps}(x_{\eps}))+(1-\eps^{2})
\left\{\half (\phi(x_{\eps}+\eps z_{\eps})-\phi(x_{\eps})) + \half (\phi(x_{\eps}-\eps z_{\eps})-\phi(x_{\eps}))\right\} +\eps^{3}.
\end{equation}
Making Taylor expansions we get 
$$ \left\{\half (\phi(x_{\eps}+\eps z_{\eps})-\phi(x_{\eps})) + \half (\phi(x_{\eps}-\eps z_{\eps})-\phi(x_{\eps}))\right\}= \half \eps^{2} \langle D^{2}\phi(x_{\eps})z_{\eps},z_{\eps}\rangle+ \textit{o}(\eps^{2}).$$
Then, from \eqref{zeps}, 
$$ 0 \leq v^{\eps}(x_{\eps})-u^{\eps}(x_{\eps})+(1-\eps^{2}) \half \langle D^{2}\phi(x_{\eps})z_{\eps},z_{\eps}\rangle+ \frac{\textit{o}(\eps^{2})}{\eps^{2}},$$
and taking the limit as $\eps \rightarrow 0$ we get 
$$ 0 \leq v(x_{0})-u(x_{0})+ \half \langle D^{2}\phi(x_{0})w_{0},w_{0}\rangle,$$
that is, 
$$-\half \Delta_{\infty}\phi(x_{0}) + u(x_{0})-v(x_{0}) \leq 0. $$

Now, if $\nabla \phi(x_{0}) =0$ we have to use the 
upper and lower semicontinuous envelopes of the equation (notice that $\Delta_\infty u$ is not well defined when
$\nabla u=0$). 
For a symmetric matrix $M \in \R^{N\times N}$ and $\xi \in \R^{N}$, we define 
$$
F_1 (\xi, M) = 
\left\{
\begin{array}{ll}
\displaystyle -\langle M \frac{\xi}{|\xi |} ; \frac{\xi}{|\xi |}  \rangle \qquad & \xi \neq 0 \\[5pt]
0  \qquad & \xi = 0
\end{array}
\right.
$$
The semicontinuous envelopes of $F_1$ are defined as 
\begin{equation}
F_1^{\ast}(\xi, M) = 
\left\{
\begin{array}{ll}
\displaystyle -\langle M \frac{\xi}{|\xi |} ; \frac{\xi}{|\xi |}  \rangle \qquad & \xi \neq 0 \\[5pt]
\displaystyle \max \Big\{ \limsup_{\eta \rightarrow 0}-\langle M \frac{\eta}{|\eta |} ; \frac{\eta}{|\eta |}  \rangle; 0 \Big\}  \qquad & \xi = 0.
\end{array}
\right.
\end{equation}
and
\begin{equation}
F_{1,\ast}(\xi, M) = 
\left\{
\begin{array}{ll}
\displaystyle -\langle M \frac{\xi}{ | \xi |} ; \frac{\xi}{|\xi |}  \rangle \qquad & \xi \neq 0 \\[5pt]
\displaystyle  \min \Big\{ \liminf_{\eta \rightarrow 0}-\langle M \frac{\eta}{|\eta |} ; \frac{\eta}{|\eta |}  \rangle; 0 \Big\}  \qquad & \xi = 0.
\end{array}
\right.
\end{equation}
Now, we just remark that 
$$ -\max_{1 \leq i \leq N} \{ \lambda_{i} \} \leq - \langle M\frac{\xi}{|\xi |},\frac{\xi}{|\xi |}\rangle \leq -\min_{1\leq i \leq N} \{ \lambda_{i} \} $$ 
and hence we obtain
\begin{equation}
F_1^{\ast}(\xi, M) = 
\left\{
\begin{array}{ll}
\displaystyle -\langle M \frac{\xi}{ |\xi |} ; \frac{\xi}{ |\xi |}  \rangle \qquad & \xi \neq 0 \\[5pt]
\displaystyle \max \Big\{ - \min_{1\leq i\leq N}\{\lambda_{i}\} ; 0 \Big\}  \qquad & \xi = 0.
\end{array}
\right.
\end{equation}
and
\begin{equation}
F_{1,\ast}(\xi, M) = 
\left\{
\begin{array}{ll}
\displaystyle -\langle M \frac{\xi}{|\xi |} ; \frac{\xi}{|\xi |}  \rangle \qquad & \xi \neq 0 \\[5pt]
\displaystyle \min \Big\{ -\max_{1\leq i\leq N}\{\lambda_{i}\}; 0 \Big\}  \qquad & \xi = 0.
\end{array}
\right.
\end{equation}

Now, let us go back to the proof and show that
\begin{equation}
\half F_{1,\ast}(0,D^{2}\phi(x_{0}))+u(x_{0})-v(x_{0}) \leq 0.
\end{equation}
As before, we have a sequence $(x_{\eps})_{\eps >0}$ such that $x_{\eps} \rightarrow x_{0}$ 
\begin{equation}
\label{ast22}
 u^{\eps}(y)-u^{\eps}(x_{\eps}) \leq \phi(y)-\phi(x_{\eps})+\eps^{3}
\end{equation}
Using the DPP, that $\phi$ is smooth and \eqref{ast22} we obtain
\begin{equation}\label{**}
0 \leq (1-\eps^{2}) \Big\{ \half \max_{\ol{B_{\eps}(x_{\eps})}}(\phi(y)-\phi(x_{\eps}))+\half \min_{\ol{B_{\eps}(x_{\eps})}}(\phi(y)-\phi(x_{\eps})) 
\Big\}+\eps^{2}(v^{\eps}(x_{\eps})-u^{\eps}(x_{\eps})) +\eps^{3} .
\end{equation}
Let $ w_{\eps} \in \ol{B_{\eps}(x_{\eps})}$ be such that 
$$ \phi(w_{\eps})-\phi(x_{\eps})=\max_{\ol{B_{\eps}(x_{\eps})}}(\phi(y)-\phi(x_{\eps})) .$$
Let $\ol{w_{\eps}}$ be the symmetric point to $w_{\eps}$ in the ball $B_{\eps}(x_{\eps})$. Then we obtain
\begin{equation}
 0 \leq (1-\eps^{2}) \Big\{ \half(\phi(w_{\eps})-\phi(x_{\eps}))+\half(\phi(\ol{w_{\eps}})-\phi(x_{\eps})) \Big\} +\eps^{2}(v^{\eps}(x_{\eps})-u^{\eps}(x_{\eps})) .
\end{equation}
Using again Taylor's expansions 
\begin{equation}
 0 \leq (1-\eps^{2})  \half\langle D^{2}\phi(x_{\eps})\frac{(w_{\eps}-x_{\eps})}{\eps},\frac{(w_{\eps}-x_{\eps})}{\eps}\rangle+v^{\eps}(x_{\eps})-u^{\eps}(x_{\eps}) + o(1) .
 \end{equation}
If for a sequence $\eps \to 0$ we have 
$$
\left | \frac{(w_{\eps}- x_{\eps})}{\eps} \right | =1,
$$
then, extracting a subsequence if necessary, we have $z\in\R^{n}$ with $\Vert z \Vert =1 $ such that
$$
\frac{(w_{\eps}- x_{\eps})}{\eps} \to z.
$$
Passing to the limit we get 
$$ 0 \leq \half \langle D^2 \phi (x_{0}) z ,
z \rangle+v(x_{0})-u(x_{0}) .$$
Then
$$-\half\max_{1\leq i\leq n}\{\lambda_{i}\}+u(x_{0})-v(x_{0}) \leq 0,$$
that is, $\half F_{1,\ast}(0,D^{2}\phi(x_{0}))+u(x_{0})-v(x_{0})\leq 0$.

Now, if we have 
$$
\left | \frac{(w_{\eps}- x_{\eps})}{\eps} \right | <1
$$
for $\eps$ small, we just observe that at those points we have that $D^2 \phi (w_{\eps})$ is negative semidefinite. 
Hence, passing to the limit we obtain that $D^2 \phi (x_0)$ is also negative semidefinite and then 
every eigenvalue of $D^2 \phi (x_0)$ is less or equal to $0$. We conclude that
$$F_{1,\ast}(0,D^2 \phi (x_{0}))=\min \{ -\max_{1\leq i \leq n}\{ \lambda_{i}\};0\}=0.$$
Moreover, for $\eps$ small we have that $\langle D^{2}\phi(x_{\eps})\frac{(w_{\eps}-x_{\eps})}{\eps},\frac{(w_{\eps}-x_{\eps})}{\eps}\rangle \leq 0$. 
Then, 
$$0\leq v^{\eps}(x_{\eps})-u^{\eps}(x_{\eps}) + o(1).$$
Taking the limit as $\eps\rightarrow 0$ we obtain  
$$u(x_{0})-v(x_{0})\leq 0.$$
Therefore we arrive to
$$\half F_{1,\ast}(x_0,D^{2}\phi(x_{0}))+u(x_{0})-v(x_{0})\leq 0,$$
that is what we wanted to show.

The fact that $u$ is a supersolution can be proved in an analogous way. In this case we
need to show that 
$$\half F_1^{\ast}(\nabla\phi(x_{0}),D^{2}\phi(x_{0}))+u(x_{0})-v(x_{0})\geq 0,$$
for $x_{0} \in \Om$ and $\phi \in {C}^{2}(\Om)$ such that $u(x_{0})-\phi(x_{0})=0$ 
and $u-\phi$ has a strict minimum at $x_{0}$.

{\it Laplacian.}
Now, let us show that $v$ is a viscosity solution to 
$$ -\frac{\kappa}{2}\Delta v(x)+v(x)-u(x)=0.$$

Let us start by showing that $u$ is a subsolution. Let $\psi \in C^{2}(\Om)$ such that $v(x_{0})-\psi(x_{0})=0$ and has a maximum of $v-\psi$ at $x_{0} \in \Om$. As before, we have the existence of a sequence $(x_{\eps})_{\eps>0}$ such that 
$x_{\eps} \rightarrow x_{0}$ and $v^{\eps}-\psi$ and
\begin{equation}
\label{ast44}
 u^{\eps}(y)-u^{\eps}(x_{\eps}) \leq \psi(y)-\psi(x_{\eps})+\eps^{3}.
\end{equation}
Therefore, from the DPP, we obtain
$$
0\leq (u^{\eps}(x_{\eps})-v^{\eps}(x_{\eps}))+(1-\eps^{2})\frac{1}{\eps^{2}}\kint_{B_{\eps}(x_{\eps})}(\psi(y)-\psi(x_{\eps}))dy.
$$
From Taylor's expansions we obtain
$$ \frac{1}{\eps^{2}}\kint_{B_{\eps}(x_{\eps})}(\psi(y)-\psi(x_{\eps}))dy=\frac{\kappa}{2} \sum\limits_{j=1}^{N} \partial_{x_{j}x_{j}}\psi(x_{\eps})=\frac{\kappa}{2}\Delta \psi(x_{\eps}),$$
with 
$\kappa = \frac{1}{\eps^{n}|B_{1}(0)|}\int_{B_{1}(0)}z_{j}^{2}\eps^{N}dz=\frac{1}{|B_{1}(0)|}\int_{B_{1}(0)}z_{j}^{2}dz. $
Taking limits as $\eps \rightarrow 0$ we get
\begin{equation}
-\frac{\kappa}{2}\Delta \psi(x_{0})+v(x_{0})-u(x_{0}) \leq 0.
\end{equation}

The fact that $v$ is a supersolution is similar.
 \end{proof}

\section{Uniqueness for viscosity solutions} \label{sect-uniqueness}

Our goal is to show uniqueness for viscosity solutions to our system \eqref{ED1}.
To this end we follow ideas from \cite{BB,Mitake} (see also \cite{Jen} for uniqueness results concerning the
infinity Laplacian). This uniqueness result implies that
the whole sequence $u^\eps,v^\eps$ converge as $\eps \to 0$. The main idea behind the proof
(as in  \cite{BB,Mitake}) is to make a change of variable $U =\psi (u)$, $V = \psi (v)$
which transforms our system \eqref{ED1} in a system in which both equations are coercive in 
their respective variables $U$ and $V$ when
$DU \neq 0$ and $DV\neq 0$. Next we use the fact that 
one can take $\psi$
as close to the identity as we want.

First, we state the Hopf Lemma. We only state the result for supersolutions (the result for subsolutions is the same with the obvious changes).

\begin{lemma} \label{Hopf-lemma}
Let $V$ be an open set with $\overline{V} \subset \Omega$. Let $(u, v)$ be a viscosity supersolution of 
\eqref{ED1} and assume that there exists $x_0 \in \partial V$ such that
$$
u (x_0) = \min \Big\{ \min_\Omega u(x); \min_\Omega v(x) \Big\} 
\qquad 
\mbox{and} \qquad 
u (x_0) < \min_{x \in V} \{u(x),v(x)\}. 
$$
Assume further that $V$ satisfies the interior ball condition at $x_0$, namely, there exists an open ball $B_R \subset V$ 
with $x_0 \in \partial B_R$. Then,
$$
\liminf_{s\to 0} \frac{ u(x_0 - s \nu (x_0)) - u(x_0)}{s} > 0,
$$
where $\nu (x_0)$ is the outward normal vector to $\partial B_R$ at $x_0$.
\end{lemma}

\begin{remark}{\rm
An analogous statement holds for the second component of the system, $v$. If we have that
$$
v (x_0) = \min \Big\{ \min_\Omega u(x); \min_\Omega v(x) \Big\} 
\qquad 
\mbox{and} \qquad
v (x_0) < \min_{x \in V} \{u(x),v(x)\}. 
$$
Then we have 
$$
\liminf_{s\to 0} \frac{ v(x_0 - s \nu (x_0)) - v(x_0)}{s} > 0.
$$
}
\end{remark}

\begin{proof}[Proof of Lemma \ref{Hopf-lemma}]
See the Appendix in \cite{Mitake}. In fact one can take
$$
w(x):=e^{-\alpha |x|^2} - e^{-\alpha R^2} 
$$
and show that $w$ is a strict subsolution of any of the two equations in \eqref{ED1} in the annulus $\{x : R/2<|x|<R\}$.
\end{proof}

The Strong Maximum Principle follows form Hopf Lemma.

\begin{theorem} \label{SMP-theo}
Let $(u, v)$ be a viscosity supersolution of \eqref{ED1}. 
Assume that $\min_\Omega \min\{ u, v\}$ is attained at an interior point of $\Omega$. Then $u = v = C$ for some constant $C$
in the whole $\Omega$.
\end{theorem}

\begin{proof}
Again we refer to the Appendix in \cite{Mitake}.
\end{proof}

Now we can proceed with the proof of the Comparison Principle.

\begin{theorem} \label{teo-compar} Assume that $(u_1, v_1)$ and $(u_2, v_2)$ are a bounded viscosity subsolution 
and a bounded viscosity supersolution of \eqref{ED1}, respectively, and also assume that $u_1 \leq u_2$ and $v_1 \leq v_2$ on 
$\partial \Omega$. Then 
$$
u_1 \leq u_2 \qquad \mbox{and} \qquad v_1 \leq v_2,
$$
in $\Omega$.
\end{theorem}

This comparison result implies the desired uniqueness for \eqref{ED1}. 

\begin{corollary} \label{corl-unicidad}
There exists a unique viscosity solution to \eqref{ED1}.
\end{corollary}

\begin{proof}[Proof of Theorem \ref{teo-compar}]
We argue by contradiction and assume that
$$
c := \max \Big\{ \max_\Omega (u_1(x) -u_2(x)) ; \max_\Omega (v_1(x) -v_2(x)) \Big\} > 0.
$$
We replace $(u_1, v_1)$ by $(u_1 - c/2, v_1- c/2)$. We may assume further that $u_1$ and $v_1$ are semi-convex 
and $u_2$ and $v_2$ are semi-concave by using sup and inf convolutions and restricting the problem to a slightly smaller domain 
if necessary (see \cite{Mitake} for extra details). 
We now perturb $u_1$ and $v_1$ as follows. For $\alpha > 0$, take $\Omega_\alpha := \{x \in \Omega : 
dist(x,\partial \Omega) > \alpha\}$ and for $|h|$ sufficiently small, define
$$
M(h):=\max \Big\{ \max_{x \in \Omega} (u_1(x+h)- u_2(x)) ; \max_{x \in \Omega} (v_1(x+h)- v_2(x))\Big\} 
 = w_1(x_h+h)- w_2(x_h) 
 $$
for $w=u \mbox{ or } v$ (we will call $w$ the component at which the maximum is achieved)
 and some $x_h \in \Omega_{|h|}$. Since $M(0) > 0$, for $|h|$ small enough, we have $M(h)>0$ and the above maximum is the same if we take it over $\Omega_\alpha$ any $\alpha >0$ sufficiently small and fixed. Note that from the equations we get that at $x_h$ we have 
 $$
 u_1(x_h + h) - u_2(x_h) = v_1(x_h + h) - v_2(x_h).
 $$
 Now, we claim that 
there exists a sequence $h_n \to 0$ such that at any maximum point $y \in \Omega_{|h_n|}$ 
 of $$\max \Big\{ \max_{x\in \Omega_{|h_n|}} (u_1(x + h_n) - u_2(x)) ;  \max_{x\in \Omega_{|h_n|}} (u_1(x + h_n) - u_2(x))\Big\},$$ 
 we have $$Dw_1(y + h_n) =
Dw_2(y) \neq 0$$ for $ n \in \mathbb{N}$. To prove this claim we argue again
 by contradiction and assume that there exists, for each $h$ with $|h|$ small, $x_h$ which is a maximum point so 
 that $Dw_1 (x_h + h) = Dw_2 (x_h) = 0$. As $u_1 - u_2$ and $v_1-v_2$ are semi-convex, $M(h)$ is semi-convex for $h$ small. 
 Now for any $k$ close to $h$, one has that, thanks to the fact that $Dw_1 (x_h + h) = 0$,
 $$
M(k) \geq w_1 (x_h+k)- w_2(x_h) \geq w_1(x_h+h)-C|h-k|^2-w(x_h) = M(h)-C|h-k|^2.
 $$
 Thus, $0 \in \partial M(h)$ for every $|h|$ small. This implies that $M(h) = M(0)$ for $|h|$ small.
Now take $x_0 \in \Omega$ a maximum point of $\max \{ \max_{x\in\Omega} (u_1(x) - u_2(x));  \max_{x\in\Omega} (v_1(x) - v_2(x))$.
For $|h|$ sufficiently small we have that $x_0 \in \Omega_{|h|}$, and, $w_1(x_0) - w_2(x_0) = M(0) = M(h) \geq w_1(x_0 + h) - w_2(x_0)$.
Hence, $x_0$ is a local maximum of $u_1$, $v_1$. Now, the strong maximum principle, Theorem \ref{SMP-theo},
implies that $u_1$, $v_1$ are constant in $\Omega$, 
which gives the desired contradiction and proves the claim.

Now we recall that for a semi-convex function $a$ and a semi-concave function $b$
we have that both $a$ and $b$ are differentiable at any local maximum points of $b-a$ and 
if the function $a$ (or $b$)  is differentiable at $x_0$ and $\{x_n\}$ is a sequence of differentiable points 
such that $x_n \to x_0$, then $Da(x_n)\to
Da(x_0)$ (or $Db(x_n)\to
Db(x_0)$). Then, thanks to these properties and the previous claim, we have the existence of a positive constant $\delta (n) > 0$ 
so that $|Dw_1(y+h_n)| = |Dw_2(y)| > \delta (n)$ for all $y$ such that the maximum in the claim is attained.

Now we consider, as in \cite{BB}, the functions $\varphi_\eps$ defined by
$$
\varphi_\eps ' (t) = \exp \left( \int_0^t \exp \Big( - \frac{1}{\eps}  (s- \frac{1}{\eps}) \Big) ds\right).
$$
These functions $\varphi_\eps$ are close to the identity, $\varphi_\eps ' > 0$, $\varphi_\eps '$ converge to $1$ as $\eps \to 0$
and $\varphi_\eps ''$ converge to $0$ as $\eps \to 0$ with $(\varphi_\eps ''(s))^2 > \varphi_\eps ''' (s) \varphi_\eps ' (s)$,
see \cite{BB}.

With $\psi_\eps = \varphi_\eps^{-1}$ we perform the changes of variables
$$
U_{i}^\eps =\psi_\eps (u_i), \qquad V_i^\eps = \psi_\eps (v_i), \qquad i=1,2.
$$
It is clear to see that $U_1$, $V_1$ are semi-convex and $U_2$, $V_2$ are semi-concave. 
We have that
$\max \{ \max_x (U_1^\eps (x+h_n)- U_2^\eps (x)) ;  \max_x (U_1^\eps (x+h_n)- U_2^\eps (x)) \}$
is achieved at some point $x_\eps$ and by passing to a subsequence if necessary, $x_\eps \to x_{h_n}$ as $\eps \to 0$. 
Since we have $|Dw_1(x_n + h_n)| = |Dw_2 (x_n)| > \delta(n)$, 
we deduce that for $\eps$ sufficiently small, it holds that $|DW_1^\eps(x_\eps + h_n)| = |DW_2^\eps (x_\eps)| \geq \delta(n)/2$.

Now, omitting the dependence on $\eps$ in what follows, we observe that, after the change of variables
$$
u_1 = \varphi (U_1), \qquad v_1 = \varphi (V_1)
$$
the pair of new unknowns $(U_1,V_1)$ verifies the equations (in the viscosity sense)
$$
\begin{array}{rl}
\displaystyle 0 & \displaystyle =- \displaystyle \half \Delta_{\infty}u_1(x) + u_1(x) - v_1(x) \\[10pt]
& \displaystyle = - \frac12 \varphi ' (U_1)\Delta_\infty U_1(x) -  \frac12 \varphi''(U_1)|DU_1|^2 (x) + \varphi(U_1(x)) - \varphi(V_1 (x))\\[10pt]
& \displaystyle = \varphi ' (U_1) \Big( - \frac12 \Delta_\infty U_1(x) -  \frac12 \frac{\varphi''(U_1)}{\varphi ' (U_1)}|DU_1|^2 (x) + 
\frac{\varphi(U_1(x)) - \varphi(V_1 (x))}{\varphi ' (U_1)}\Big),
\end{array}
$$
and 
$$
\begin{array}{rl}
\displaystyle 0 & \displaystyle = -   \displaystyle \frac{\kappa}{2} \Delta v(x) + v(x) - u(x) \\[10pt]
& =  \displaystyle - \frac{\kappa}{2}  \Big( \varphi ' (V_1)\Delta V_1(x) +  \varphi''(V_1)|DV_1|^2 (x) \Big) + \varphi(V_1 (x)) - \varphi(U_1 (x) ) \\[10pt]
& =  \displaystyle  \varphi ' (V_1) \Big(- \frac{\kappa}{2} \Delta V_1(x) - \frac{\kappa}{2}   \frac{\varphi''(V_1)}{\varphi ' (V_1)}
|DV_1|^2 (x) + \frac{\varphi(V_1 (x)) - \varphi(U_1 (x) )}{ \varphi ' (V_1)}  \Big),
\end{array}
$$
and similar equations also hold for $(U_2,V_2)$.

Since $|DU_1^\eps(x_\eps + h_n)| = |DV_2^\eps (x_\eps)| \geq \delta(n)/2$
this system is strictly monotone (the first equation is monotone in $U_1$ and the second in $V_1$. Here we use that $(\varphi_\eps ''(s))^2 > \varphi_\eps ''' (s) \varphi_\eps ' (s)$ that implies that
$\varphi_\eps '' /\varphi_\eps '$ is increasing, i.e. $-(\varphi_\eps '' /\varphi_\eps ')' >0$. Thus, 
from the strict monotonicity, we get the desired contradiction. See the proof of \cite{BB}, Lemma 3.1, for a more detailed discussion.
\end{proof}

\section{Possible extensions of our results} \label{sect.extensiones}

In this section we gather some comments on more general systems that can be
studied using the same techniques.

\subsection{Coefficients with spacial dependence}

We can look at the case in which the probability of jumping from one board to the other
depends on the spacial location, that is, we can take the probability to jump from board 1 to 2 
as $a(x) \eps^2$ and from 2 to 1 as $b(x) \eps^2$, for two given nonnegative functions $a(x)$, $b(x)$. 
In this case the DPP is given by 
$$
 \left\lbrace
\begin{array}{ll}
 \displaystyle u^{\eps}(x)= a(x) \eps^{2}v^{\eps}(x)+(1- a(x) \eps^{2})\Big\{\half \sup_{y \in B_{\eps}(x)}u^{\eps}(y) + \half \inf_{y \in B_{\eps}(x)}u^{\eps}(y)
 \Big\} \qquad &  \ x \in \Omega,  \\[10pt]
   \displaystyle v^{\eps}(x)= b(x) \eps^{2}u^{\eps}(x)+(1- b(x) \eps^{2})\kint_{B_{\eps}(x)}v^{\eps}(y)dy  \qquad & \ x \in \Omega,  \\[10pt]
u^{\eps}(x) = \ol{f}(x) \qquad & \ x \in \R^{n} \backslash \Omega,  \\[10pt]
v^{\eps}(x) = \ol{g}(x) \qquad & \ x \in \R^{n} \backslash \Omega. 
\end{array}
\right.
$$ 
and the limit system is 
$$
 \left\lbrace
\begin{array}{ll}
- \displaystyle \half \Delta_{\infty}u(x) + a(x) u(x) - a(x) v(x)=0 \qquad &  \ x \in \Omega,  \\[10pt]
-   \displaystyle \frac{\kappa}{2} \Delta v(x) + b(x) v(x) - b(x) u(x)=0  \qquad &  \ x \in \Omega,  \\[10pt]
u(x) = f(x) \qquad & \ x \in \partial \Omega,  \\[10pt]
v(x) = g(x) \qquad & \ x \in \partial \Omega, 
\end{array}
\right.
$$

\subsection{$n\times n$ systems}
We can deal with a system of $n$ equations $N$ unknowns, $u_1,...,u_n$, 
of the form
$$
 \left\lbrace
\begin{array}{ll}
- \displaystyle  L_i u_i(x) + b_i u_i(x) - \sum_{j\neq i} a_{ij} u_j(x)=0 \qquad &  \ x \in \Omega,  \\[10pt]
u_i(x) = f_i(x) \qquad & \ x \in \partial \Omega.
\end{array}
\right.
$$
Here $L_i$ is  $\Delta_\infty$ or $\Delta$, and the coefficients $b_i$, $a_{ij}$ are nonnegative
and verify
$$
b_i = \sum_{j\neq i} a_{ij}.
$$

To handle this case we have to play in $n$ different boards and
take the probability of jumping from board $i$ to board $j$ as $a_{ij} \eps^2$ 
(notice that then the probability of continue playing in the same board
$i$ is 
$1- \sum_{j\neq i} a_{ij} \eps^2$).
The associated DPP is given by 
$$
 \left\lbrace
\begin{array}{l}
 \displaystyle u_i^{\eps}(x)= \eps^{2} \sum_{j\neq i} a_{ij} u_j^{\eps}(x)+(1- b_i \eps^{2}) 
 \Big\{\half \sup_{y \in B_{\eps}(x)}u_i^{\eps}(y) 
 + \half \inf_{y \in B_{\eps}(x)}u_i^{\eps}(y)
 \Big\}  \\
 \mbox{ or }  \\
  \displaystyle u_i^{\eps}(x)= \eps^{2} \sum_{j\neq i} a_{ij} u_j^{\eps}(x)+(1- b_i \eps^{2}) 
  \kint_{B_{\eps}(x)}u_i^{\eps}(y)dy  \qquad\qquad\qquad \qquad \ x \in \Omega,  \\[10pt]
 u_i^{\eps}(x) = \ol{f}(x) \qquad \qquad \ x \in \R^{N} \backslash \Omega. 
\end{array}
\right.
$$

\subsection{Systems with normalized $p-$Laplacians}
The normalized $p-$Laplacian is given by
$$
\Delta_p^N u (x) = \alpha \Delta_\infty u (x) + \beta \Delta u(x),
$$
with $\alpha (p)$, $\beta (p)$ verifying 
$
\alpha + \beta =1$
(see \cite{MPRb}). Notice that this operator is $1-$homogeneous.
With the same ideas used here we can also handle the system
$$
 \left\lbrace
\begin{array}{ll}
- \displaystyle  \Delta_{p}^N u(x) + u(x) - v(x)=0 \qquad &  \ x \in \Omega,  \\[10pt]
- \displaystyle \Delta_{q}^N v(x) + v(x) - u(x)=0 \qquad &  \ x \in \Omega,  \\[10pt]
u(x) = f(x) \qquad &  \ x \in \partial \Omega,  \\[10pt]
v(x) = g(x) \qquad & \ x \in \partial \Omega.
\end{array}
\right.
$$

The associated game runs as follows, in the first board, 
when the token does not jump, a biased coin is towed (with probabilities $\alpha(p)$ os heads and
 $\beta(p)$ of tails), if we get heads then we play Tug-of-War and if we get tails then we move at random.
 In the second board the rules are the same but we use a biased coin with different probabilities 
 $\alpha(q)$ and $\beta(q)$, see \cite{MPRa}, \cite{MPRb}, \cite{PS} and \cite{R} for a similar game for a scalar
 equation (playing in only one board). The corresponding DPP is:
$$
 \left\lbrace
\begin{array}{l}
 \displaystyle u^{\eps}(x)=\eps^{2}v^{\eps}(x)+(1-\eps^{2}) \left[ \alpha (p)
 \Big\{\half \sup_{y \in B_{\eps}(x)}u^{\eps}(y) 
 + \half \inf_{y \in B_{\eps}(x)}u^{\eps}(y)
 \Big\}  +  \beta(p) \kint_{B_{\eps}(x)}u^{\eps}(y)dy \right] \qquad  \ x \in \Omega,  \\[10pt]
  \displaystyle v^{\eps}(x)=\eps^{2}u^{\eps}(x)+(1-\eps^{2}) \left[ \alpha (q)
 \Big\{\half \sup_{y \in B_{\eps}(x)}v^{\eps}(y) 
 + \half \inf_{y \in B_{\eps}(x)}v^{\eps}(y)
 \Big\}  +  \beta(q) \kint_{B_{\eps}(x)}v^{\eps}(y)dy \right] \qquad  \ x \in \Omega,  \\[10pt]
u^{\eps}(x) = \ol{f}(x) \qquad  \ x \in \R^{N} \backslash \Omega, \\[10pt]
v^{\eps}(x) = \ol{g}(x) \qquad  \ x \in \R^{N} \backslash \Omega. 
\end{array}
\right.
$$

\medskip

{\bf Acknowledgements.} Partially supported by CONICET grant PIP GI No 11220150100036CO
(Argentina), by  UBACyT grant 20020160100155BA (Argentina) and by MINECO MTM2015-70227-P
(Spain).



\begin{thebibliography}{BH}


\bibitem{TPSS} T. Antunovic, Y. Peres, S. Sheffield and S. Somersille. {\it Tug-of-war and infinity Laplace equation with vanishing Neumann boundary condition}. Comm. Partial Differential Equations, 37(10), 2012, 1839--1869.

\bibitem{Arroyo} A. Arroyo and J. G. Llorente. {\it On the asymptotic mean value property for planar p-harmonic functions}. Proc. Amer. Math. Soc. 144 (2016), no. 9, 3859--3868.


\bibitem{AS} S. N. Armstrong and C. K. Smart. {\it An easy proof of Jensen's theorem on the uniqueness of infinity harmonic functions.}
Calc. Var. Partial Differential Equations 37(3-4) (2010), 381--384.

\bibitem{BB} G. Barles and J. Busca. {\it Existence and comparison results for fully nonlinear degenerate elliptic equations without zeroth-order term.} Comm. Partial Differential Equations 26 (2001), no. 11-12, 2323--2337.

\bibitem{BR} P. Blanc and J. D. Rossi. {\it Games for eigenvalues of the Hessian and conca- ve/convex envelopes.} 
J. Math. Pures et Appliquees. 127, (2019), 192--215.

\bibitem{BRLibro} P. Blanc and J. D. Rossi. {Game Theory and Partial Differential Equations.}
De Gruyter Series in Nonlinear Analysis and Applications  Vol. 31. 2019.
ISBN 978-3-11-061925-6.
ISBN 978-3-11-062179-2 (eBook).

\bibitem{ChGAR} F. Charro, J. Garcia Azorero and J. D. Rossi. {\it A mixed problem for
the infinity laplacian via Tug-of-War games.} Calc. Var. Partial
Differential Equations, 34(3), (2009), 307--320.

\bibitem{Cran} M.G. Crandall. {\it A visit with the $\infty$-Laplace equation}, Calculus of variations and nonlinear partial differential equations, Lecture Notes in Math., vol. 1927, Springer, Berlin, 2008, pp. 75--122.


\bibitem{CIL} M.G. Crandall,  H. Ishii and  P.L. Lions. {\it User's guide to viscosity solutions of second order partial differential equations}. Bull. Amer. Math. Soc. 27 (1992), 1--67.\label{CIL}

\bibitem{Doob} J.L. Doob, {\it What is a martingale ?}, Amer. Math. Monthly, 78 (1971), no. 5, 451--463.

\bibitem{I} M. Ishiwata, R. Magnanini and H. Wadade. {\it A natural approach to the asymptotic mean value property for the p-Laplacian}. Calc. Var. Partial Differential Equations, 56 (2017), no. 4, Art. 97, 22 pp.

\bibitem{Jen} R. Jensen. {\it Uniqueness of Lipschitz extensions: minimizing the sup norm of the gradient.} Arch. Rational Mech. Anal. 123 (1993), no. 1, 51--74.

\bibitem{Kac} M. Kac. {\it Random Walk and the Theory of Brownian Motion}.
Amer. Math. Monthly, 54, No. 7, (1947), 
369--391.

\bibitem{KMP} B. Kawohl, J.J. Manfredi and M. Parviainen. {\it Solutions of nonlinear PDEs in the sense of averages}. J. Math. Pures Appl. 97(3), (2012), 173--188.


\bibitem{LM} P. Lindqvist and J. J. Manfredi. {\it On the mean value property for the $p-$Laplace equation in the plane}. Proc. Amer. Math. Soc. 144 (2016), no. 1, 143--149.


\bibitem{QS}
Q. Liu and A. Schikorra.
\textit{General existence of solutions to dynamic programming principle.}
 Commun. Pure Appl. Anal. 14 (2015), no. 1, 167--184.
 

\bibitem{LPS}
H. Luiro, M. Parviainen, and E. Saksman.
{\it Harnack's inequality for p-harmonic functions via stochastic games.}
Comm. Partial Differential Equations, 38(11), (2013), 1985--2003




\bibitem{MPR} J. J. Manfredi, M. Parviainen and J. D. Rossi. {\it An asymptotic mean value characterization for p-harmonic functions}. Proc. Amer. Math. Soc. 138 (2010), no. 3, 881--889. 

 \bibitem{MPRa} J. J. Manfredi, M. Parviainen and J. D. Rossi.
\textit{Dynamic programming principle for tug-of-war games with noise.}
ESAIM, Control, Opt. Calc. Var., 18, (2012), 81--90.

\bibitem{MPRb} J. J. Manfredi, M. Parviainen and J. D. Rossi.
\textit{On the definition and properties of p-harmonious functions.}
Ann. Scuola Nor. Sup. Pisa, 11, (2012), 215--241.

\bibitem{Mitake} H. Mitake and H. V. Tran, {\it 
Weakly coupled systems of the infinity Laplace equations}, Trans. Amer. Math. Soc. 369 (2017), 1773--1795.


\bibitem{PSSW} Y. Peres, O. Schramm, S. Sheffield and D. Wilson,
{\it Tug-of-war and the infinity Laplacian.} J. Amer. Math. Soc.,
22, (2009), 167--210.


\bibitem{PS} Y. Peres and S. Sheffield, {\it Tug-of-war with noise:
a game theoretic view of the $p$-Laplacian}, Duke Math. J., 145(1), (2008), 91--120.

\bibitem{R} J. D. Rossi. {\it Tug-of-war games and PDEs.} Proc.
Royal Soc. Edim. 141A, (2011), 319--369.

\bibitem{Williams} D. Williams, Probability with martingales, Cambridge University Press, Cambridge,
1991.




\end{thebibliography}
\end{document}